\documentclass[12pt,authoryear,letterpaper,reqno,oneside]{amsart}
\usepackage[left=1.5in,right=1.5in,bottom=1.5in,top=1.5in]{geometry}
\usepackage{amsmath,amsfonts,amssymb, amsthm}%
\usepackage{mathrsfs}
\usepackage{stmaryrd}
\usepackage[usenames,dvipsnames]{xcolor}
\usepackage{multicol}
\usepackage{setspace}
\usepackage{comment}
\usepackage{graphicx}
\usepackage{url}
\usepackage{xparse, expl3}

\usepackage{hyperref}
\hypersetup{
	colorlinks = true,
	citecolor=DarkBlue,
	urlcolor=DarkBlue,
	linkcolor=Green
}
\usepackage{hypernat}

\usepackage[capitalize]{cleveref}  
\crefname{lemma}{Lemma}{Lemmas}
\crefname{corollary}{Corollary}{Corollaries}
\crefname{theorem}{Theorem}{Theorems}
\crefname{equation}{Equation}{Equations}
\crefname{example}{Example}{Examples}
\crefname{section}{Section}{Sections}
\crefname{subsection}{Section}{Sections}

\usepackage[%
    minnames=1,maxnames=99,maxcitenames=4,
    style=alphabetic,
    doi=false,
    isbn=false,
    url=false,
    firstinits=true,
    hyperref,
    natbib,
    backend=bibtex]{biblatex}
\bibliography{bibliography}

\DeclareDocumentCommand{\SetStructure}{m m m}
{
	\expandafter\DeclareDocumentCommand\csname strUctUrE#1fIEld#2\endcsname{} {#3}
}
\DeclareDocumentCommand{\GetStructure}{m m}
{
	\csname strUctUrE#1fIEld#2\endcsname
}

\def\ClassUndErscOrE{UndErscOrE}

\ExplSyntaxOn % or spaces would count

\cs_new_protected:Nn \Class_Code:nn
{
	\cs_set_protected:Nn \__Class_Code:n { #2 }
	\__Class_Code:n { #1 }
}

\NewDocumentCommand{\ClassAddObject}{m m m +m}
{
	\exp_args:Nc \DeclareDocumentCommand{#1\ClassUndErscOrE#2}{#3}{#4}
}

\DeclareDocumentCommand{\ClassSet}{m m +m}
{
	\ClassAddObject{#1}{#2}{}{#3}
}

\DeclareDocumentCommand{\ClassAddClass}{m m m}
{
	\csname Class#2\endcsname{#3\ClassUndErscOrE}

	\ClassSet{#1}{#3}
	{
	\csname#3\ClassUndErscOrE\endcsname
	}
}

\NewDocumentCommand{\Class}{m +m}
{
	\expandafter\NewDocumentCommand\csname Class#1\endcsname{m}
	{
		
		\expandafter\DeclareDocumentCommand\csname ##1\endcsname {d..}
		{
			\csname##1\ClassUndErscOrE####1\endcsname
		}

		\ClassAddObject{##1}{Set}{m m}
		{
			\ClassSet{##1}{####1}{####2}
		}

		\ClassSet{##1}{Class}{#1}
		
		\ClassSet{##1}{Self}{##1}

		\Class_Code:nn {##1}{#2}

	}
}

\ExplSyntaxOff

\DeclareDocumentCommand{\PushStack}{m m}
{
	\stepcounter{stAckcOUntEr#1}
	\SetStructure{stAckstrUctUrE#1}{\arabic{stAckcOUntEr#1}}{#2}	
}

\DeclareDocumentCommand{\PopStack}{m}
{
	\ReadStack{#1}
	\addtocounter{stAckcOUntEr#1}{-1}	
}

\DeclareDocumentCommand{\PopCounterStack}{m}
{
	\addtocounter{stAckcOUntEr#1}{-1}	
}

\DeclareDocumentCommand{\ReadStack}{m}
{
	\GetStructure{stAckstrUctUrE#1}{\arabic{stAckcOUntEr#1}}
}

\DeclareDocumentCommand{\DeclareStack}{m}
{
	\newcounter{stAckcOUntEr#1}
	\addtocounter{stAckcOUntEr#1}{-1}
	\expandafter\DeclareDocumentCommand\csname Push#1\endcsname {m}
		{\PushStack{#1}{##1}}
	\expandafter\DeclareDocumentCommand\csname Pop#1\endcsname {} 
		{\PopStack{#1}}
	\expandafter\DeclareDocumentCommand\csname Read#1\endcsname {} 
		{\ReadStack{#1}}
	\expandafter\DeclareDocumentCommand\csname Counter#1\endcsname {} 
		{\arabic{stAckcOUntEr#1}}
}

\DeclareDocumentCommand{\DeclareCounter}{m}%
{%
 	\ifcsname c@#1\endcsname%
	\else%
		\newcounter{#1}%
	\fi%
}

\DeclareDocumentCommand{\ForLoopcOUntErdEfInEd}{m D[]{0} d[] D(){1} d<>} 
{%
\ifnum\value{#1}<#3%
	{%
	#5%
	\addtocounter{#1}{#4}%
	\ForLoopcOUntErdEfInEd{#1}[#2][#3](#4)<#5>%
	}%
\fi%
}
\DeclareDocumentCommand{\ForLoop}{m D[]{0} d[] D(){1} d<>} 
{%
	\DeclareCounter{#1}%
	\setcounter{#1}{#2}%
	\ForLoopcOUntErdEfInEd{#1}[#2][#3](#4)<#5>%
}

\ExplSyntaxOn

\DeclareDocumentCommand{\DeclareMathOperatorBrackets}{m m}{%
\expandafter\DeclareMathOperator\csname #1mAthOpErAtOrdUmmYvArIAblE\endcsname{#2}
\expandafter\DeclareDocumentCommand\csname #1\endcsname{}{\EM{{\csname #1mAthOpErAtOrdUmmYvArIAblE\endcsname}}}
}

\ExplSyntaxOff

\ExplSyntaxOn

\DeclareDocumentCommand{\DeclareFunctionSymbol}{m m o}{%
\IfNoValueTF{#3}
	{
	\expandafter\DeclareDocumentCommand\csname #1\endcsname{m}{
		#2(##1)
		}
	}
	{
	\expandafter\DeclareDocumentCommand\csname #1\endcsname{d()}{
		\IfNoValueTF{##1}
			{
			#3
			}
			{
			#2(##1)
			}
		}

	}
}

\ExplSyntaxOff

\DeclareDocumentCommand{\Lomega}{m}{{\EM{\mc{L}_{#1, \w}}}}
\DeclareDocumentCommand{\Lww}{}{\Lomega{\w}}
\DeclareDocumentCommand{\Lwow}{}{\Lomega{\w_1}}

\DeclareDocumentCommand{\defn}{m}{{\bf{#1}}}
\DeclareDocumentCommand{\defas}{}{\EM{:=}}

\DeclareDocumentCommand{\st}{}{\EM{\colon}}

\DeclareDocumentCommand{\:}{}{\colon}

\DeclareDocumentCommand{\w}{}{\EM{\omega}}
\DeclareDocumentCommand{\Naturals}{}{{\EM{{\mbb{N}}}}}
\DeclareDocumentCommand{\Nats}{}{\Naturals}

\DeclareDocumentCommand{\NatsIndex}{}{\EM{{\Nats}^{[<\w]}}}

\DeclareMathOperatorBrackets{MathOpS}{S}
\DeclareDocumentCommand{\SymmetricGroupSymbol}{}{\EM{\mathfrak{S}}}
\DeclareDocumentCommand{\Sym}{m}{\EM{\SymmetricGroupSymbol_{#1}}}

\DeclareDocumentCommand{\SymDiff}{}{\EM{\triangle}}

\DeclareDocumentCommand{\And}{}{\EM{\wedge}}

\DeclareDocumentCommand{\<}{}{\EM{\langle}}
\DeclareDocumentCommand{\>}{}{\EM{\rangle}}

\DeclareDocumentCommand{\qu}{m}{``#1''}

\DeclareDocumentCommand{\nl}{}{\newline}

\DeclareDocumentCommand{\M}{}{{\EM{\mathcal{M}}}}
\DeclareDocumentCommand{\N}{}{{\EM{\mathcal{N}}}}

\DeclareDocumentCommand{\B}{}{{\EM{\mathbb{B}}}}

\DeclareDocumentCommand{\X}{}{{\EM{\mathbb{X}}}}

\DeclareDocumentCommand{\X}{}{{\EM{\mc{X}}}}

\DeclareDocumentCommand{\x}{}{{\EM{\ol{x}}}}
\DeclareDocumentCommand{\y}{}{{\EM{\ol{y}}}}
\DeclareDocumentCommand{\z}{}{{\EM{\ol{z}}}}
\DeclareDocumentCommand{\a}{}{{\EM{\ol{a}}}}
\DeclareDocumentCommand{\b}{}{{\EM{\ol{b}}}}
\DeclareDocumentCommand{\c}{}{{\EM{\ol{c}}}}

\DeclareDocumentCommand{\e}{}{{\EM{\ol{e}}}}
\DeclareDocumentCommand{\f}{}{{\EM{\ol{f}}}}
\DeclareDocumentCommand{\g}{}{{\EM{\ol{g}}}}

\DeclareDocumentCommand{\k}{}{{\EM{\ol{k}}}}

\DeclareDocumentCommand{\z}{}{{\EM{\ol{z}}}}

\DeclareDocumentCommand{\n}{}{{\EM{\textbf{n}}}}

\DeclareDocumentCommand{\cs}{}{\text{cs}}

\DeclareMathOperatorBrackets{ZF}{ZF}
\DeclareMathOperatorBrackets{ZFC}{ZFC}
\DeclareMathOperatorBrackets{Set}{SET}
\DeclareDocumentCommand{\Powerset}{}{\mathfrak{P}}

\DeclareMathOperatorBrackets{ORD}{ORD}
\DeclareMathOperatorBrackets{dom}{dom}
\DeclareMathOperatorBrackets{range}{range}
\DeclareMathOperatorBrackets{len}{len}
\DeclareMathOperatorBrackets{obj}{obj}
\DeclareMathOperatorBrackets{Obj}{Obj}
\DeclareMathOperatorBrackets{morph}{morph}
\DeclareMathOperatorBrackets{Morph}{Morph}
\DeclareMathOperatorBrackets{im}{im}
\DeclareMathOperatorBrackets{cod}{cod}
\DeclareMathOperatorBrackets{codom}{codom}
\DeclareMathOperatorBrackets{graph}{graph}
\DeclareMathOperatorBrackets{rank}{rank}
\DeclareMathOperatorBrackets{Hom}{Hom}
\DeclareMathOperatorBrackets{id}{id}

\DeclareDocumentCommand{\Fraisse}{}{\text{Fra\"{i}ss\'{e}}}

\DeclareDocumentCommand{\Erdos}{}{\text{Erd\H{o}s}}
\DeclareDocumentCommand{\Renyi}{}{\text{R\'{e}nyi}}

\DeclareDocumentCommand{\Lang}{}{\mathrm{L}}
\DeclareDocumentCommand{\Theory}{}{\text{Th}}
\DeclareDocumentCommand{\Th}{}{\Theory}

\DeclareMathOperatorBrackets{Sort}{Sort}
\DeclareMathOperatorBrackets{Rel}{Rel}
\DeclareMathOperatorBrackets{Funct}{Funct}

\DeclareMathOperatorBrackets{arity}{ar}

\DeclareMathOperatorBrackets{dedcl}{dc}

\DeclareMathOperatorBrackets{QuantifierFree}{qf}

\DeclareMathOperatorBrackets{qftp}{qtp}
\DeclareMathOperatorBrackets{nqftp}{ntp}
\DeclareMathOperatorBrackets{oqftp}{otp}
\DeclareDocumentCommand{\qfpi}{}{\QuantifierFree_\pi}
\DeclareMathOperatorBrackets{tp}{tp}
\DeclareMathOperatorBrackets{etp}{etp}
\DeclareMathOperatorBrackets{itp}{itp}
\DeclareMathOperatorBrackets{gtp}{gtp}
\DeclareMathOperatorBrackets{TP}{TP}
\DeclareMathOperatorBrackets{ITP}{ITP}
\DeclareMathOperatorBrackets{GTP}{GTP}

\DeclareMathOperatorBrackets{Stab}{Stab}
\DeclareMathOperatorBrackets{Age}{Age}

\DeclareDocumentCommand{\ScottSen}{m}{\EM{\sigma_{#1}}}

\DeclareDocumentCommand{\Rado}{}{\EM{\mathscr{R}}}
\DeclareDocumentCommand{\Trado}{}{\EM{\mathscr{T}}}

\DeclareDocumentCommand{\CanStr}{m}{\mathfrak{C}_{#1}}

\DeclareMathOperatorBrackets{Models}{Mod}
\DeclareMathOperatorBrackets{Str}{Str}
\DeclareMathOperatorBrackets{Aut}{Aut}

\DeclareDocumentCommand{\extent}{m d()}{
\IfNoValueTF{#2}
	{
	{\EM{\llbracket{#1}\rrbracket}}
	}
	{
	{\EM{\llbracket{#2}\rrbracket_{#1}}}
	}
}
\DeclareDocumentCommand{\Extent}{m d()}{
\IfNoValueTF{#2}
	{
	{\EM{\left\llbracket{#1}\right\rrbracket}}
	}
	{
	{\EM{\left\llbracket{#2}\right\rrbracket_{#1}}}
	}
}

\DeclareMathOperatorBrackets{RankForClosure}{rank}

\DeclareDocumentCommand{\leqclosure}{m}{{\EM{\sqsubseteq_{#1}}}}
\DeclareDocumentCommand{\eqclosure}{m}{{\EM{\square_{#1}}}}
\DeclareDocumentCommand{\ltclosure}{m}{{\EM{\sqsubset_{#1}}}}
\DeclareDocumentCommand{\nleqclosure}{m}{{\EM{\not\sqsubseteq_{#1}}}}
\DeclareDocumentCommand{\neqclosure}{m}{{\EM{\not\!\square_{#1}}}}
\DeclareDocumentCommand{\nltclosure}{m}{{\EM{\not\sqsubset_{#1}}}}
\DeclareDocumentCommand{\rankclosure}{m}{{\EM{\RankForClosure_{#1}}}}
\DeclareDocumentCommand{\imclosure}{m m}{{\EM{[#2]_{\eqclosure{#1}}}}}

\ExplSyntaxOn

\DeclareDocumentCommand{\DeclareClosure}{m o o}{%
\IfNoValueTF{#2}
	{%
	\expandafter\DeclareMathOperator\csname #1tExtdUmmYvArIAblE\endcsname{#1}
	}
	{%
	\expandafter\DeclareMathOperator\csname #1tExtdUmmYvArIAblE\endcsname{#2}
	}
\IfNoValueTF{#3}
	{%
	\IfNoValueTF{#2}
		{%
		\expandafter\DeclareMathOperator\csname #1sUbscIptdUmmYvArIAblE\endcsname{#1}
		}
		{%
		\expandafter\DeclareMathOperator\csname #1sUbscIptdUmmYvArIAblE\endcsname{#2}
		}
	}
	{%
	\expandafter\DeclareMathOperator\csname #1sUbscIptdUmmYvArIAblE\endcsname{#3}
	}

\expandafter\DeclareDocumentCommand\csname #1\endcsname{}{\csname #1tExtdUmmYvArIAblE \endcsname}

\expandafter\DeclareDocumentCommand\csname leq#1\endcsname{}{\leqclosure{\csname #1sUbscIptdUmmYvArIAblE\endcsname}}

\expandafter\DeclareDocumentCommand\csname eq#1\endcsname{}{\eqclosure{\csname #1sUbscIptdUmmYvArIAblE\endcsname}}

\expandafter\DeclareDocumentCommand\csname lt#1\endcsname{}{\ltclosure{\csname #1sUbscIptdUmmYvArIAblE\endcsname}}

\expandafter\DeclareDocumentCommand\csname nleq#1\endcsname{}{\nleqclosure{\csname #1sUbscIptdUmmYvArIAblE\endcsname}}

\expandafter\DeclareDocumentCommand\csname neq#1\endcsname{}{\neqclosure{\csname #1sUbscIptdUmmYvArIAblE\endcsname}}

\expandafter\DeclareDocumentCommand\csname nlt#1\endcsname{}{\nltclosure{\csname #1sUbscIptdUmmYvArIAblE\endcsname}}

\expandafter\DeclareMathOperator\csname rank#1\endcsname{\rankclosure{\csname #1sUbscIptdUmmYvArIAblE\endcsname}}

\expandafter\DeclareDocumentCommand\csname im#1\endcsname{ m }{\imclosure{\csname #1sUbscIptdUmmYvArIAblE\endcsname}{##1}}
	
}
\ExplSyntaxOff

\DeclareClosure{cl}
\DeclareClosure{acl}
\DeclareClosure{dcl}
\DeclareClosure{tdcl}

\DeclareDocumentCommand{\eqd}{}{\EM{\overset{d}{=}}}
\DeclareMathOperatorBrackets{indep}{\EM{\perp\!\!\!\perp}}
\DeclareMathOperatorBrackets{as}{ a.s. }

\DeclareDocumentCommand{\ProbMeasures}{}{\EM{\mathcal{P}_1}}

\DeclareMathOperatorBrackets{Terms}{Terms}
\DeclareMathOperatorBrackets{Ind}{Ind}

\DeclareMathOperatorBrackets{Shell}{Shell}
\DeclareMathOperatorBrackets{Tail}{Tail}

\DeclareDocumentCommand{\overlay}{m d()}{
\IfNoValueTF{#2}
	{
	\EM{\mathfrak{O}_{#1}}
	}
	{
	\EM{\mathfrak{O}_{#1}(#2)}
	}
}

\DeclareFunctionSymbol{templatedUmmYvArIAblE}{TEMPLATE}[TEMPLATE]%
\DeclareDocumentCommand{\template}{o m}{
\IfNoValueTF{#1}
	{
	\templatedUmmYvArIAblE(#2)
	}
	{
	\templatedUmmYvArIAblE(#1;#2)
	}
}

\DeclareFunctionSymbol{SigmaTemplatedUmmYvArIAblE}{\EM{\Sigma}}[\EM{\Sigma}]%
\DeclareDocumentCommand{\SigmaTemplate}{o m}{
\IfNoValueTF{#1}
	{
	\SigmaTemplatedUmmYvArIAblE(#2)
	}
	{
	\SigmaTemplatedUmmYvArIAblE(#1;#2)
	}
}

\DeclareDocumentCommand{\k}{}{\EM{\mathbf{k}}}

\DeclareDocumentCommand{\Str}{}{
\EM{\mathscr{Str}}
}

\DeclareMathOperatorBrackets{triv}{triv}

\DeclareDocumentCommand{\NonRedTh}{m}{\mathscr{T}_{#1}}

\DeclareDocumentCommand{\FreeCS}{m}{\EM{\mathfrak{F}(#1)}}
\DeclareDocumentCommand{\RandMod}{m}{\EM{\mathscr{M}(#1)}}
\DeclareDocumentCommand{\StrNR}{}{\EM{\Str^*}}

\DeclareDocumentCommand{\EM}{m}{\ensuremath{#1}}
\DeclareDocumentCommand{\RightJustify}{m}{\hspace*{\fill}\mbox{#1}\penalty-9999\relax}

\DeclareDocumentCommand{\mbb}{m}{\EM{\mathbb{#1}}}

\DeclareDocumentCommand{\mc}{m}{\EM{\mathcal{#1}}}

\DeclareDocumentCommand{\ol}{m}{\EM{\overline{#1}}}

\DeclareDocumentCommand{\ul}{m}{\underline{#1}}

\definecolor{RealGreen}{rgb}{0,1,0}
\definecolor{ActualGreen}{rgb}{0.0,0.5,0.0}
\definecolor{DarkGreen}{rgb}{0, 0.3,0}

\definecolor{Blue}{rgb}{0,0,1}
\definecolor{DarkBlue}{rgb}{0,0,0.6}

\definecolor{Red}{rgb}{1,0,0}
\definecolor{DarkRed}{rgb}{0.5,0,0}

\DeclareDocumentEnvironment{testing}{d() D[]{\MyQED} d<>}
{%BEGIN
\noindent\IfNoValueTF{#1}
{%
\emph{Proof.\!\!\!\!}
}
{
\emph{Proof\ #1.\ }
}
}
{%END
\IfValueTF{#3}
	{%	
	\ExplSyntaxOff
	\RightJustify{\EM{#2_{#3}}}
	\vspace{1ex}
	}
	{
	\RightJustify{\EM{#2}}
	\vspace{1ex}
	}
}

\DeclareDocumentCommand{\MyQED}{}{\qed}

\DeclareDocumentEnvironment{proof}{d() D[]{\MyQED} d<>}
{%BEGIN
\noindent\IfNoValueTF{#1}
{\emph{Proof.\!\!}}
{\emph{Proof\ #1.\ }}
}
{%END
\IfValueTF{#3}
	{%	
	\ExplSyntaxOff
	\RightJustify{\EM{#2_{#3}}}
	\vspace{1ex}
	}
	{
	\RightJustify{\EM{#2}}
	\vspace{1ex}
	}
}

\DeclareCounter{ProofLabelcOUntEr}
\DeclareDocumentCommand{\ProofLabel}{}{%
\addtocounter{ProofLabelcOUntEr}{1}
\label{cUrrEntProoflAbEl\arabic{ProofLabelcOUntEr}}
}

\DeclareDocumentCommand{\ProofRef}{D<>{1}}
{%
\ref{cUrrEntProoflAbEl\arabic{ProofcOUntEr#1}}
}
\DeclareDocumentCommand{\ProofCref}{D<>{1}}
{%
\cref{cUrrEntProoflAbEl\arabic{ProofcOUntEr#1}}
}

\DeclareDocumentEnvironment{Proof}{d() D[]{\MyQED} D<>{1}}
{%BEGIN
\DeclareCounter{ProofcOUntEr#3}%
\setcounter{ProofcOUntEr#3}{\value{ProofLabelcOUntEr}}%
\begin{proof}(#1)[#2]<\ProofRef<#3>>%
}
{%END
\end{proof}
}

\ExplSyntaxOn
\def\TheoremDepth{section}

\DeclareDocumentCommand{\DeclareTheorem}{m o m o}{%
\IfNoValueTF{#4}
	{%
	\IfNoValueTF{#2}
		{%
		\newtheorem{#1vArIAblE}{#3}
		}
		{%
		\newtheorem{#1vArIAblE}[#2vArIAblE]{#3}
		}
	}
	{%
	\newtheorem{#1vArIAblE}{#3}[#4]%
	}
\newtheorem*{#1vArIAblE*}{#3}

\DeclareDocumentEnvironment{#1}{o o}
	{%BEGIN
	\IfValueT{##2}%
		{
		\begin{spacing}{##2}
		}
	\IfValueTF{##1}
		{
		\begin{#1vArIAblE}[##1]
		}
		{
		\begin{#1vArIAblE}
		}
	\ProofLabel
	}
	{%END
	\IfValueT{##2}%
		{
		\end{spacing}{##2}
		}
	\end{#1vArIAblE}
	}

\DeclareDocumentEnvironment{#1*}{o o}
	{%BEGIN
	\IfValueT{##2}%
		{
		\begin{spacing}{##2}
		}
	\IfValueTF{##1}
		{
		\begin{#1vArIAblE*}[##1]
		}
		{
		\begin{#1vArIAblE*}
		}
	}
	{%END
	\IfValueT{##2}%
		{
		\end{spacing}{##2}
		}
	\end{#1vArIAblE*}
	}
}
\ExplSyntaxOff

\theoremstyle{plain}
\DeclareTheorem{theorem}{Theorem}[\TheoremDepth]
\DeclareTheorem{proposition}[theorem]{Proposition}
\DeclareTheorem{lemma}[theorem]{Lemma}
\DeclareTheorem{corollary}[theorem]{Corollary}
\DeclareTheorem{claim}[theorem]{Claim}
\DeclareTheorem{subclaim}[theorem]{Subclaim}

\theoremstyle{definition}
\DeclareTheorem{definition}[theorem]{Definition}
\DeclareTheorem{axiom}[theorem]{Axiom}
\DeclareTheorem{notation}[theorem]{Notation}
\DeclareTheorem{convention}[theorem]{Convention}
\DeclareTheorem{example}[theorem]{Example}

\theoremstyle{remark}
\DeclareTheorem{remark}[theorem]{Remark}
\DeclareTheorem{question}[theorem]{Question}
\DeclareTheorem{conjecture}[theorem]{Conjecture}

\begin{document}

\title%[Aut(M)-Invariant Measures]
{Representations of Aut(M)-Invariant Measures}

\author[Ackerman]{Nathanael Ackerman}
\address{Harvard University\\
Cambridge, MA 02138
\\ USA}
\email{nate@aleph0.net}

\begin{abstract}
We study probability measures on the space of countable models which are invariant under the group of automorphisms of a fixed structure. We show that for a class of structures, which we call \emph{free}, every such invariant measure has a representation in terms of a simple collection of measurable functions, akin to the representation given in the Aldous-Hoover-Kallenberg theorem. We then use this to give a characterization of those probability measures which are invariant under the automorphism group of an arbitrary structure, not necessarily free, and which have such a representation. 
\end{abstract}

%%%%%%%%%%%%%%%%%%%%%%%%%%%%%%%%

\maketitle
\thispagestyle{empty}

\vspace*{-20pt}

\begin{small}
\setcounter{tocdepth}{2}
\tableofcontents
\end{small}

\vspace*{-20pt}

%\newpage

%%%%%%%% INPUT/INCLUDES %%%%%%%%

%% START %%
%\input{start}
%\maketitle
%\clearpage

%% PAPER %%

%%%%  %%%%  %%%%  %%%%  %%%%  %%%%  %%%%  %%%%
%%%%  %%%%  %%%%  %%%%  %%%%  %%%%  %%%%  %%%%

%%%%  %%%%  %%%%  %%%%  %%%%  %%%%  %%%%  %%%%
\section{Introduction}
%%%%  %%%%  %%%%  %%%% %%%%  %%%%  %%%%  %%%%
By studying the symmetries of a collection of random variables one can gain great insight into their collective structure. In particular, one is often able to deduce properties of a collection of random variables simply by knowing that its joint distribution is preserved under a certain group of symmetries (and indeed much of ergodic theory can be recast in this light). Over the last century, there have been several \emph{representation} results for collections of random variables whose joint distributions are preserved under various, sufficiently large, groups. These representation results show that if a collection of random variables has a joint distribution which is preserved under an appropriate group of symmetries (usually the group of permutations of $\Nats$ or variants of it) then the random variables must have a very simple form. 

As an example, we now consider those collections of random variables which are closed under an action of $\Sym{\Nats}$, the group of permutations of $\Nats$.

\begin{definition}
For a standard Borel space $S$, a sequence of $S$-valued random variables $(X_n)_{n \in \Nats}$ is said to be \emph{exchangeable} if its joint distribution is unchanged by permutations of $\Nats$, i.e.\ for all $\sigma \in \Sym{\Nats}$
\[
(X_n)_{n \in \Nats} \eqd (X_{\sigma(n)})_{n \in \Nats}. 
\]
\end{definition}

A natural class of exchangeable sequences consists of those which are of the form 
\[
(f(\zeta_\emptyset, \zeta_n))_{n \in \Nats}
\]
where $f\:[0,1]^2 \to S$ is any measurable function and $\{\zeta_\emptyset\} \cup \{\zeta_n\}_{n \in \Nats}$ is a collection of uniform identically distributed and independent (i.i.d.) $[0,1]$-valued random variables. An important result of de Finetti and Hewitt--Savage is that all exchangeable sequences have a distribution which is of this form.  This result, known as \emph{de Finetti's theorem}, is a fundamental result of modern probability theory with applications throughout Bayesian nonparametric statistics, genetics, statistical mechanics, and ergodic theory (see, e.g. \cite{MR0494344}, \cite{MR883646}, \cite{MR2161313} Ch.~1.1).  

Let $\NatsIndex$ be the collection of finite sequences of distinct natural numbers and let $\Powerset(\cdot)$ and $\Powerset_{<\w}(\cdot)$ be the powerset and finite powerset operations respectively. 

\begin{definition}
For a standard Borel space $S$, an array of $S$-valued random variables $(X_\a)_{\a \in \NatsIndex}$ is said to be \emph{(jointly) exchangeable} if its distribution is unchanged under permutations of $\Nats$, i.e.\ if for all $\sigma \in \Sym{\Nats}$
\[
(X_\a)_{\a \in \NatsIndex} \eqd (X_{\sigma(\a)})_{\a \in \NatsIndex}. 
\]
\end{definition}

The following is a natural class of exchangeable arrays. For $n \in \Nats$ let $f_n\:[0,1]^{\Powerset([n])}\to S$ and let $(\zeta_\a)_{\a\in \Powerset_{<\w}(\Nats)}$ be a collection of uniform i.i.d.\ $[0,1]$-valued random variables. For $\a \in \NatsIndex$ let $\widehat{\zeta}_\a = (\zeta_\b)_{\b \in \Powerset(\a)}$. The following is then an exchangeable array 
\[
(f_{|\a|}(\widehat{\zeta}_\a))_{\a \in \NatsIndex}.
\]

An important result is that every exchangeable array is equivalent in distribution to one of the above form. This theorem was proved by Aldous, Hoover, and Kallenberg in various levels of generality; for a survey see \cite{MR2463438}, and for a detailed history, see the Historical Notes of \cite{MR2161313}. This result can be seen as a higher dimensional version of de Finetti's theorem, and is known as the \emph{Aldous-Hoover-Kallenberg theorem (for jointly exchangeable arrays)}.  This theorem has also been rediscovered in various forms, notably in terms of the notion of a \emph{graphon}, which can be thought of as the limit of a sequence of dense graphs. For more on this connection see \cite{MR2426176} or \cite{MR2463439}.

The focus in this paper will be on probability measures, and hence on the distributions of random variables. In particular we call the distribution of an exchangeable sequence or array $\Sym{\Nats}$-invariant as it is preserved by all elements of $\Sym{\Nats}$ under the appropriate action. We will be interested in the distribution of arrays which, instead of being preserved under all elements of $\Sym{\Nats}$, only need to be preserved under elements of a closed subgroup of $\Sym{\Nats}$.

It is well known that a subgroup of $\Sym{\Nats}$ is closed if and and only if it is the automorphism group of a countable structure with underlying set $\Nats$. Now if $\M$ is such a structure then there is a natural collection of arrays whose distribution is preserved under all automorphisms of $\M$, i.e. elements of $\Aut(\M)$. Specifically for $\a \in \M$ let $p_\a$ be the orbit of $\a$ under $\Aut(\M)$ and for any such orbit let $f_{p_\a}\:[0,1]^{\Powerset(|\a|)}\to S$ be a measurable function. It is easy to check that the array $(f_{p_\a}(\widehat{\zeta}_\a))_{\a \in \NatsIndex}$ is $\Aut(\M)$-invariant. We call the distribution of such an array \emph{representable}.

A key result of this paper is a characterization of the $\Aut(\M)$-invariant measures which are representable, where $\M$ is an arbitrary structure in a countable language with underlying set $\Nats$. Specifically, we introduce in \cref{Free Structure} the notion of a \emph{free} structure  and produce explicit $\Sym{\Nats}$-invariant measures concentrated on free structures. We then give a general method of combining measures which will allow us, when $\M_{fr}$ is free, to combine an $\Aut(\M_{fr})$-invariant measure with a $\Sym{\Nats}$-invariant measure we construct which is concentrated on $\M_{fr}$. We will then obtain a measure to which we can apply the Aldous-Hoover-Kallenberg theorem. This will then let us deduce in \cref{Aut(M)-invariant measures for M free are representable} that whenever $\M_{fr}$ is free, all $\Aut(\M_{fr})$-invariant measures are representable. We then show that for all countable structures $\M$, there is a minimal free extension $\FreeCS{\M}$ and that an $\Aut(\M)$-invariant measure is representable if and only if it can be extended to an $\Aut(\FreeCS{\M})$-invariant measure (in a sense we will make precise).

These results not only generalize the Aldous-Hoover-Kallenberg theorem to a large class of closed subgroups of $\Sym{\Nats}$, but also provides a way to characterize the representable measures even when such an Aldous-Hoover-Kallenberg like representation theorem fails to hold.

%%%%  %%%% %%%%  %%%%  %%%%  %%%%
\subsection{Connections To Other Work}
%%%%  %%%% %%%%  %%%%  %%%%  %%%%

In this subsection we consider how several related results can be viewed as special cases of our representation theorem. 

Suppose $\M_\emptyset$ is the unique structure on $\Nats$ in the empty language. Then $\Aut(\M_\emptyset) = \Sym{\Nats}$, and the Aldous-Hoover-Kallenberg theorem says that all $\Aut(\M_\emptyset)$-invariant measures are representable. As $\M_\emptyset$ is free (a fact which is obvious from \cref{Free Structure}) the Aldous-Hoover-Kallenberg theorem can be seen as a special case of \cref{Aut(M)-invariant measures for M free are representable}, which says whenever $\M$ is free all $\Aut(\M)$-invariant measures are representable.

In addition to considering arrays whose distributions are invariant under a single, simultaneous, permutation of the indices, Aldous and Hoover also considered arrays whose distributions are invariant when the indices are permuted separately. 

\begin{definition}
An array of $S$-valued random variables $(X_\a)_{\a \in \Nats^n}$ is said to be \emph{separately exchangeable} if its distribution is unchanged under coordinatewise permutations of $\Nats$, i.e.\ if for all $\sigma_0, \dots, \sigma_{n-1} \in \Sym{\Nats}$
\[
(X_{\<a_i\>_{i < n}})_{\<a_i\>_{i < n} \in \Nats^n}  \eqd (X_{\<\sigma_i(a_i)\>_{i < n}})_{\<a_i\>_{i < n} \in \Nats^n}. 
\]
\end{definition}

As for jointly exchangeable arrays Aldous and Hoover proved a representation theorem for separately exchangeable arrays of arbitrary (fixed) dimension. However, to simplify the presentation we will restrict ourselves in the following example to the case when $n =2$. Both the Aldous-Hoover theorem for separately exchangeable arrays and the example below generalizes in the obvious way for $n > 2$. For more on separately exchangeable arrays see \cite{MR2161313}.

The Aldous-Hoover theorem for separately exchangeable arrays can be recast in terms of \cref{Aut(M)-invariant measures for M free are representable} as follows.

\begin{example}
Suppose $(X_{a_0, a_1})_{(a_0, a_1) \in \Nats^2}$ is an $S$-valued array. Aldous and Hoover showed that $(X_{a_0, a_1})_{(a_0, a_1) \in \Nats^2}$ is separately exchangeable if and only if there are uniform i.i.d.\ $[0,1]$-valued random variables $\beta, (\zeta_a)_{a\in \Nats}, (\eta_a)_{a\in \Nats}$ and $(\gamma_{i,j})_{i, j \in \Nats}$ along with a measurable function $f\:[0,1]^4 \to S$ such that 
\[
(X_{a_0,a_1})_{a_0,a_{1} \in \Nats} \eqd (f(\beta, \zeta_{a_0}, \eta_{a_1}, \gamma_{a_0, a_1}))_{a_0, a_1 \in \Nats}.
\]
Let $\Lang = \{U_0, U_1\}$  be a language with two unary relations and let $\M_{U, 2}$ be the structure with underlying set $\Nats$ such that $U_0^\M = \{2n \st n \in \Nats\}$ and $U_1^\M = \{2n+1 \st n \in \Nats\}$. Automorphisms of $\M_{U, 2}$ can be represented as pairs $(\sigma, \tau) \in \Sym{\Nats}^2$ where the pair $(\sigma,\tau)(2n) = 2\sigma(n)$ and $(\sigma, \tau)(2n+1) = 2\tau(n) + 1$. Next consider the array $(Y_{2a_0, 2a_1+1})_{(2a_0, 2a_1+1) \in \Nats^2}$ where $Y_{2a_0, 2a_1+1} \defas X_{a_0, a_1}$. Then 
\[
(X_{a_0, a_1})_{a_0, a_1 \in \Nats}
\]
is separately exchangeable if and only if 
\[
(Y_{2a_0, 2a_1+1})_{(2a_0, 2a_1+1) \in \Nats^2}
\]
has an $\Aut(\M_{U, 2})$-invariant distribution. Notice that all of the indices of the array 
\[
(Y_{2a_0, 2a_1+1})_{(2a_0, 2a_1+1) \in \Nats^2}
\]
are in the same $\Aut(\M_{U,2})$-orbit. Therefore the Aldous-Hoover theorem for separately exchangeable arrays is equivalent to the statement that all arrays indexed by elements in the orbit of the pair $\<0,1\>$ in $\M_{U,2}^2$, and which are $\Aut(\M_{U,2})$-invariant, are representable. It is easily seen from  \cref{Free Structure} that $\M_{U,2}$ is free. Therefore the Aldous-Hoover theorem for separately exchangeable arrays can be seen as a special case of \cref{Aut(M)-invariant measures for M free are representable}. 
\end{example}

Independently of the present work, Crane and Towsner (\cite{MR3835071})
have studied the case of $\Aut(\M)$-invariant measures on the collection of
structures in a fixed finite relational language with universe $\Nats$, when $\M$ is freeand ultrahomogeneous. Their notation and framework though are slightly different than ours. In particular, the notion of \emph{relative exchangeability} which they introduce is, for a non-ultrahomogeneous structure $\M$, more restrictive than $\Aut(\M)$-invariance. However, in \cite{MR3835071} they restrict their attention to ultrahomogeneous structures and, as they observe, in this setting the two notions coincide. Further their notion of an ultrahomogeneous structure whose age has $n$-DAP for all $n \geq 1$ is precisely that of a free ultrahomogeneous structure in our sense. 

With these translations Theorem~3.2 of \cite{MR3835071}, which is one of their two main representation theorems, is equivalent to a finitary version of \cref{Aut(M)-invariant measures for M free are representable}. Specifically it is, essentially, \cref{Aut(M)-invariant measures for M free are representable} restricted to the case where the probability measure is invariant under the action of the automorphisms of a structure which is ultrahomogeneous in a finite relational language (i.e. $\Aut(\M)$ where $\Lang_\M$ is finite) and is on a space of structures in a finite relational language (i.e. $\Str_\Lang(\M)$ with $\Lang$ finite). One of the main advantages of \cref{Aut(M)-invariant measures for M free are representable} over Theorem~3.2 of \cite{MR3835071} is that by allowing the language of the free structure $\M$ to be infinite we are able to apply \cref{Aut(M)-invariant measures for M free are representable} to arbitrary free structures by first passing to the canonical structure (which is never in a finite language). This in turn allows us to characterize representable measures over any structure, as in \cref{Equivalence of extension to free structure and representability}.

%%%%  %%%%  %%%%  %%%%  %%%%  %%%%
\subsection{Outline}
%%%%  %%%% %%%%  %%%%  %%%%  %%%%

We now give an outline of the rest of the paper. First, in \cref{Notation Subsection} we collect the notation which we will need. In \cref{Preliminaries Section} we give background and preliminary material which will be important later. In particular in \cref{Polish Group Actions Subsection} we introduce the notion of a canonical structure and review basic facts from the study of Polish group actions. In \cref{Infinitary Logic Subsection} we recall fundamental ideas from the study of $\Lwow$ and introduce the important notions of non-redundant quantifier free types as well as ordered quantifier free types. In \cref{Definable Expansion Subsection} we introduce the notion of definable expansions and show that for every fragment and every language there is a theory (in a different language) which admits quantifier elimination for the fragment, which is non-redundant, and which is interdefinable with the empty theory in the original language. In \cref{Invariant Measures Subsection} we introduce the notion of an invariant measure. These are the main object of study in this paper. In  \cref{Aldous-Hoover-Kallenberg Subsection} we recall the Aldous-Hoover-Kallenberg theorem as well as related results. These results are a key building block of the main results of this paper.  We then end the preliminaries with \cref{Existence of S-Nat-Invariant Measures Subsection} where we recall the necessary and sufficient conditions for a sentence of $\Lwow$ to admit an $\Sym{\Nats}$-invariant measure. 

In \cref{Free Structures Section} we prove various results related to free structures. Specifically, in \cref{Canonical Structures Subsection} we continue our discussion of canonical structures and introduce a notion of map between canonical structures which will be important. Then, in \cref{Free Structures Subsection}, we introduce the notion of a free structure, show that every canonical structure is contained in a minimal free structure, and give an explicit example, for each free structure, of a $\Sym{\Nats}$-invariant measure concentrated on its isomorphism class. 

In \cref{Merging of Measures Section} we introduce the notion of merging measures. Specifically we show that if $\Aut(\M) \subseteq \Aut(\N)$ then we can combine (in a unique way) an $\Aut(\N)$-invariant measure on the space of $\Lang_\M$-structures and concentrated on the isomorphism class of $\M$ with an $\Aut(\M)$-invariant measure on the space of $\Lang$-structures to get an $\Aut(\N)$-invariant measure on the space of $\Lang$-structures which is concentrated on the isomorphism class of $\M$. Further every $\Aut(\N)$-invariant measure on the space of $\Lang$-structures which is concentrated on the isomorphism class of $\M$ arises as such a combination. We will use this result to take an $\Aut(\M_{fr})$-invariant measure, when $\M_{fr}$ is free, and combine it with the measure we constructed in \cref{Free Structures Subsection} to get a $\Sym{\Nats}$-invariant measure to which we can apply the Aldous-Hoover-Kallenberg theorem. In \cref{Inherited Properties Of Merged Measures Subsection} we will consider properties the merged measures inherit from their parts as well as give a characterization, for any free structure $\M_{fr}$ of those sentences of $\Lwow$ which admit an $\Aut(\M_{fr})$-invariant measure.

In \cref{Representations Section} we put everything together to study the representability of $\Aut(\M)$-invariant measures. In \cref{Aut(M)-Recipes Subsection} we introduce the notion of an $\Aut(\M)$-recipe. Representable measures will be exactly those which are the distribution of an $\Aut(\M)$-recipe. In \cref{Representability and Free Structures Subsections} we prove two of the main results of this paper. First we show that if $\M_{fr}$ is free then every $\Aut(\M)$-invariant measure is representable. Then we show that for any structure $\M$ an $\Aut(\M)$-invariant measure is representable if and only if it has an extension to an $\Aut(\FreeCS{\M})$-invariant measure where $\FreeCS{\M}$ is the minimal free structure containing $\M$. In \cref{Ergodic Representations Subsection} we give a characterization of when an $\Aut(\M)$-recipe is gives rise  to an ergodic measure.

%%%%  %%%% %%%%  %%%%  %%%%  %%%%
\subsection{Notation}
\label{Notation Subsection}
%%%%  %%%% %%%%  %%%%  %%%%  %%%%

For $n \in \Nats$ we let $[n] = \{0, \dots, n-1\}$ and let $[\w] = \Nats$. 
We will use $\NatsIndex$ to denote the collection of finite sequences of distinct natural numbers and let $\Nats^{[<d]}$ be the collection of finite sequences of distinct natural numbers of length less than $d$. 

We let $\Powerset(X)$ denote the collection of all subsets of $X$.
For $n \leq \w$ and $\boxdot \in \{=, < , \leq\}$ we let $\Powerset_{\boxdot n}(X)$ be the collection of subsets of $X$ which have cardinality $\boxdot n$.  When $\a$ is a sequence and no confusion can arise we will abuse notation and let $\Powerset(\a)$ denote the set of subsequences of $\a$. 

For $\k = (k_0, \dots, k_{d-1}) \in \NatsIndex$ and $I = \{i_1, \dots, i_m\} \in \Powerset(d)$ with $i_1 < \cdots < i_m$ we let $k \circ I = \{k_{i_1}, \dots, k_{i_m}\}$. Similarly for $K = \{k_0, \dots, k_{d-1}\} \in \Powerset_{d}(\Nats)$ with $k_0 < \dots < k_{d-1}$ and $I = \{i_1, \dots, i_m\} \in \Powerset(d)$ with $i_1 < \cdots < i_m$ we let $K \circ I = \{k_{i_1}, \dots, k_{i_m}\}$.
If $(E_\a)_{\a \in \Powerset_{<\w}(\Nats)}$ is an indexed collection of objects and $\b \in \NatsIndex$ we let $\widehat{E}_\b \defas \<E_{\b \circ I}\>_{I \in \Powerset(|\b|)}$ and if $B \in \Powerset_{<\w}(\Nats)$ we let $\widehat{E}_B \defas \<E_{B \circ I}\>_{I \in \Powerset(|\b|)}$. 

If $\equiv$ is an equivalence relation on a set $X$ and $x \in X$ let $[x]_{\equiv}$ be the $\equiv$-equivalence class of $x$. If $A, B$ are sets we let $A \SymDiff B$ be the symmetric difference of $A$ and $B$.

All languages will be countable and relational. Note that by a standard interpretation of functions by their graphs, restricting to relational languages yields no loss of generality. Further $\Lang$ and its variants with decorations will always represent languages. If $R$ is a relation we let $\arity(R)$ be its arity. We denote by $\Lang^n$ the sub-language of $\Lang$ consisting of those relations of arity exactly $n$ and by $\Lang^{\leq n}$ be the sublanguage of $\Lang$ consisting of those relations of arity at most $n$. We let $\Lww(\Lang)$ be the collection of first order formulas in the language $\Lang$ and we let $\Lwow(\Lang)$ be the collection of infinitary formulas in the language $\Lang$ with at most countable sized conjunctions and disjunctions. All formulas will be in $\Lwow(\Lang)$ for some language $\Lang$. For a formula $\varphi$, and $i \in \Nats$, we let $\neg^i \varphi$ stand for $\neg \varphi$ if $i$ is odd and $\varphi$ if $i$ is even. For a formula $\varphi(\x, y) \in \Lwow(\Lang)$ we let $(\exists^{=1} y)\varphi(\x, y)$ stand for ``there exists a unique $y$ satisfying $\varphi(\x, y)$''. 

The symbol $\M$ and its variants with decorations will always be countable structures for some language and we will use $\M$ for both the structure and the underlying set when no confusion can arise. We will use $\Lang_\M$ to denote the language of $\M$. When $\a$ is a finite tuple of elements from $\M$ we will abuse notation and write $\a \in \M$ to denote $\a \in \M^{|\a|}$. 

Suppose $\x = \<x_i\>_{i \in [n]}$ where $n \leq \w$. We define a function $\gamma_\x\: [0,1] \to \x$ as follows. First $\gamma_\x(1) \defas x_0$. If $n < \w$ and $y\in [0,1)$ then $\gamma_{\x}(y) \defas x_i$ if and only if $y \in [\frac{i}{n}, \frac{i+1}{n})$. If $n = \w$ and $y \in [0,1)$ then $\gamma_\x(y) \defas x_i$ if and only if $y \in [1 - 2^{-i}, 1 - 2^{-(i+1)})$. 

We say $(\zeta_i)_{i \in I}$ is a \defn{$U[0,1]$-array} if $(\zeta_i)_{i \in I}$ consists of i.i.d.\ $[0,1]$-valued random variables with uniform distribution. We let $\lambda$ be the Lebesgue measure. For random variables $X, Y$ taking values in the same space we let $X \eqd Y$ denote the fact that $X$ and $Y$ have the same distribution. We will use \qu{\as} to denote the phrase \qu{almost surely}. 

All measures in this paper will be probability measures and all spaces will be standard Borel spaces. Further $S$ and its variants with decorations will always be standard Borel spaces and we let $\ProbMeasures(S)$ be the collection of probability measures on $S$. 

For a set $X$ we denote by $\Sym{X}$ the collection of permutations of $X$. We will consider $\Sym{\Nats}$ as a Polish group with the subspace topology inherited from $\Nats^\Nats$. If $\M$ is an $\Lang$-structure we denote by $\Aut(\M)$ the collection of automorphisms of $\M$. 

For any notions of probability theory not explicitly covered here we refer the reader to \cite{MR1876169}. For any notions of model theory not covered here we refer the reader to \cite{MR0424560}. For any notions of descriptive set theory we refer the reader to \cite{MR1321597} or \cite{MR1425877}.

%%%%  %%%%  %%%%  %%%% %%%%  %%%%  %%%%  %%%%
\section{Preliminaries}
\label{Preliminaries Section}
%%%%  %%%%  %%%%  %%%% %%%%  %%%%  %%%%  %%%%

In this section we recall some important facts and results which will be used later.

%%%%  %%%% %%%%  %%%%  %%%%  %%%%
\subsection{Polish Group Actions}
\label{Polish Group Actions Subsection}
%%%%  %%%% %%%%  %%%%  %%%%  %%%%

Suppose $G$ is a closed subgroup of $\Sym{\Nats}$. For $\a, \b \in \NatsIndex$ let $\a \sim_G \b$ if there is a $g \in G$ such that $\a = g(\b)$. It is immediate that $\sim_G$ is an equivalence relation on $\NatsIndex$ in which $\sim_G$-equivalent tuples have the same length. 

Let $\Lang_G \defas \{R_{[\a]_{\sim_G}}(\x) \st \a \in \NatsIndex\}$ where $\arity(R_{[\a]_{\sim_G}}) = |\a|$ for $\a \in \NatsIndex$. We call $\Lang_G$ the \defn{canonical language} of $G$. Now let $\CanStr{G}$ be the $\Lang_G$-structure with underlying set $\Nats$ such that $\CanStr{G} \models R_{A}(\b)$ if and only if $\b \in A$. We call $\CanStr{G}$ the \defn{canonical structure} of $G$. The following two lemmas are then immediate. 

\begin{lemma}[\cite{MR1425877} Sec.~1.5]
\label{Info about canonical structure of a group}
If $G$ is any closed subgroup of $\Sym{\Nats}$ then 

\begin{itemize}
\item $G = \Aut(\CanStr{G})$.

\item $\CanStr{G}$ is the canonical structure of $\Aut(\CanStr{G})$.  

\item $\CanStr{G}$ is ultrahomogeneous, i.e.\ any isomorphism between finite structures extends to an automorphism. 

\end{itemize}
\end{lemma}

\begin{lemma}
If $\M$ is a structure with underlying set $\Nats$ then $\Aut(\M)$ is a closed subgroup of $\Sym{\Nats}$. 
\end{lemma}

In particular for the purposes of studying closed subgroups of $\Sym{\Nats}$ it suffices to restrict our attention to groups of the form $\Aut(\M)$ where $\M$ is the $\Aut(\M)$-canonical structure. This is significant because there is a concrete representation of actions of $G$ for $G$ a closed subgroup of $\Sym{\Nats}$ in terms of its canonical structure. 

\begin{definition}
Suppose $G$ is a Polish group. A \defn{$G$-space} is a pair $(\circ_X, X)$ where
\begin{itemize}
\item $X$ is a Borel space. 

\item $\circ\: G \times X \to X$ is a Borel map. 

\end{itemize}

A function $i\:(\circ_X, X) \to (\circ_Y, Y)$ is a map of $G$-spaces if it is a Borel function such that $(\forall g \in G)(\forall x \in X)\ i(\circ_X(g, x)) = \circ_Y(g, i(x))$.  
\end{definition}

If $(\circ_X,X)$ is a $G$-space then we extend the action of $G$ to subsets of $X$ where, for $A \subseteq X$ and $g \in G$, $gA \defas \{\circ_X(g, a)\st a \in A\}$. 

\begin{definition}
Suppose $\M$ is an $\Lang_\M$-structure with underlying set $\Nats$ and suppose $\Lang$ is disjoint from $\Lang_\M$. We define $\Str_{\Lang}(\M)$ to be the collection of $\Lang_\M \cup \Lang$ structures with underlying set $\Nats$ such that whenever $\N \in \Str_\Lang(\M)$ then $\N|_{\Lang_\M} = \M$. 

For $k \in \Nats$ (possibly $0$), $\varphi(x_0, \dots, x_{k-1})\in \Lwow(\Lang_\M \cup \Lang)$ and $n_0, \dots, n_{k-1} \in \Nats$ define $\extent{\varphi(n_0, \dots, n_{k-1})}_\M \defas \{\N \in \Str_\Lang(\M) \st \N \models \varphi(n_0, \dots, n_{k-1})\}$. 

We give $\Str_\Lang(\M)$ the topology generated by the clopen subbasis 
\[
\{\extent{R(n_0, \dots, n_{k-1})}_\M\st R \in \Lang, \arity(R) = k, n_0, \dots, n_{k-1} \in \Nats\}
\]

\end{definition}

When $\M_\emptyset$ is the unique structure on $\Nats$ in the empty language we will denote $\Str_\Lang(\M_\emptyset)$ by $\Str_\Lang$. In this case observe that $\Aut(\M_\emptyset) = \Sym{\Nats}$. 

Now for any $\M$, there is a natural action of $\Aut(\M)$ on $\Str_\Lang(\M)$. 

\begin{definition}
Suppose $\M$ is an $\Lang_\M$-structure with underlying set $\Nats$ and $\Lang$ is a language disjoint from $\Lang_\M$. We define the action $\circ_\M\: \Aut(\M) \times \Str_\Lang(\M) \to \Str_\Lang(\M)$ where, for $g \in \Aut(\M)$, $\N \in \Str_\Lang(\M)$, $\circ_\M(g, \N)$ is the structure $g\N$ such for all $R \in \Lang$ of arity $k$ and $n_0, \dots, n_{k-1} \in \Nats$
\[
g\N \models R(n_0, \dots, n_{k-1})\text{ if and only if } \N \models R(g^{-1}(n_0), \dots, g^{-1}(n_{k-1}))
\]
\end{definition}

It is immediate that $(\circ_\M, \Str_\Lang(\M))$ is an $\Aut(\M)$-space. 

\begin{lemma}
\label{Str_L(M) is universal for Aut(M) actions}
Suppose $\M$ is an $\Lang_\M$-structure with underlying set $\Nats$ and $\Lang$ is a language disjoint from $\Lang_\M$. Further suppose $\Lang$ has relations of unbounded arity. Then $(\circ_\M, \Str_\Lang(\M))$ is a universal $\Aut(\M)$-space, i.e.\ an $\Aut(\M)$-space which contains an isomorphic copy of every other $\Aut(\M)$-space as a subspace.

\end{lemma}
\begin{Proof}
\cite{MR1425877} Thm.~2.7.4.
\end{Proof}

In particular \cref{Info about canonical structure of a group} and \cref{Str_L(M) is universal for Aut(M) actions} tell us that if $G$ is a closed subgroup of $\Sym{\Nats}$ then the study of $G$-invariant measures on $G$-spaces is equivalent to the study of $\Aut(\M)$-invariant measures on $\Str_\Lang(\M)$. This is significant as it allows us to translate the study of $G$-invariant measures from the realm of descriptive set theory to the realm of model theory. 

\begin{definition}
We say a Borel set $A\subseteq \Str_\Lang(\M)$ is \defn{$\Aut(\M)$-invariant} if for all $g \in \Aut(\M)$, $gA = A$.
\end{definition}

Note this notion is sometimes called \emph{strict invariance} to contrast it with \cref{Almost sure invariance and ergodic}. 

It is immediate that if $\tau \in \Lwow(\Lang_\M \cup \Lang)$ is a sentence then $\extent{\tau}_\M$ is an $\Aut(\M)$-invariant Borel subset of $\Str_\Lang(\M)$ and hence $\extent{\tau}_\M$ inherits the structure of an $\Aut(\M)$-space.

\begin{definition}
We say a sentence $T \in \Lwow(\Lang_\M \cup \Lang)$ is \defn{$\Aut(\M)$-universal} if $\extent{T}_\M$ is a universal $\Aut(\M)$-space.

\end{definition}

We will often want to assume our models satisfy some basic syntactic properties, e.g.\ non-redundancy of relations, Morleyized for a fragment, etc. Provided we can find an $\Aut(\M)$-universal theory whose models are exactly those with the desired syntactic properties, then there is no loss in generality in assuming our structures satisfy those properties. We will come back to this in \cref{Definable Expansion Subsection}.

%%%%  %%%% %%%%  %%%%  %%%%  %%%%
\subsection{Infinitary Logic}
\label{Infinitary Logic Subsection}
%%%%  %%%% %%%%  %%%%  %%%%  %%%%

In this section we recall some basic facts and definitions from infinitary logic. In particular it will be important in what follows to pin down various notions of quantifier free type. First we recall some basic properties of structures.

\begin{lemma}[\cite{MR0424560} Ch.~VII.6]
Suppose $\M$ is a countable $\Lang_\M$-structure. Then there is a sentence $\ScottSen{\M}\in \Lwow(\Lang_\M)$, called the \defn{Scott sentence} of $\M$, such that for each countable $\Lang_\M$-structure $\N$
\[
\N \models \ScottSen{\M}\text{ if and only if }\M \cong \N.
\]
\end{lemma}

We will often want to focus on structures $\M$ which are, in some sense, far from being rigid (and hence will have a large automorphism group). One way to express this is by saying the structure has trivial definable closure. 

\begin{definition}
For $\a \in \M$ the \defn{(group theoretic) definable closure} of $\a$ is the set 
\[
\dcl(\a) \defas \{b \in \M \st (\forall g \in \Aut(\M))\ g(\a) = \a \rightarrow g(b) = b\}.
\]
We say $\M$ has \defn{trivial definable closure}, or \defn{trivial dcl}, if 
\[
(\forall \a \in \M)\ \dcl(\a) = \a.
\] 
\end{definition}

We will often want to restrict attention to subsets of $\Lwow(\Lang)$ with basic closure properties.

\begin{definition}
A \defn{fragment} is a subset of $A \subseteq \Lwow(\Lang)$ which is closed under
\begin{itemize}
\item Sub-formulas,

\item finite Boolean operations, and

\item $(\exists x), (\forall x)$. 

\end{itemize}

An \defn{$A$-theory} is a collection of sentences of the form $\{\varphi \in A \st \M \models \varphi\}$ for some $\Lang$-structure $\M$. 

\end{definition}

In addition to having a notion of trivial dcl for a structure, we will also want an analogous notion for a theory in a fragment. 

\begin{definition}
Suppose $A$ is a fragment and $T$ is an $A$-theory. We say $T$ has \defn{trivial definable closure} (or trivial dcl) if there does not exist a formula $\varphi(\x, y) \in A$ such that 
\[
T \models (\exists \x)(\exists^{=1} y)\varphi(\x, y)
\] 
\end{definition}

Note that a structure $\M$ has trivial dcl if and only if for any fragment $A$, any $A$-theory containing $\ScottSen{\M}$ has trivial dcl. Similarly a structure $\M$ has trivial dcl if and only if there is some fragment $A$ and some $A$-theory containing $\ScottSen{\M}$ which as trivial dcl. 

\begin{definition}
We say a sentence $T \in \Lwow(\Lang)$ is \defn{Morleyized} for a fragment $A$ if for all $\varphi(\x) \in A$ there is a relation $R_\varphi(\x) \in \Lang$ such that $T\models (\forall \x)\varphi(\x) \leftrightarrow R_\varphi(\x)$. 
\end{definition}

We now introduce some important notions involving quantifier free types.

\begin{definition}
A \defn{partial quantifier free $\Lang$-type on $(x_0, \dots, x_{n-1})$} is a collection of formula, $q$, such that whenever $\eta(x_{i_0}, \dots, x_{i_{k-1}}) \in q$ we have
\begin{itemize}
\item $\{x_{i_0}, \dots, x_{i_{k-1}}\} \subseteq \{x_0, \dots, x_{n-1}\}$, 

\item $\eta$ is either an atomic formula or the negation of an atomic formula, 

\item for all $0 \leq i < j <n$, $x_i \neq x_j \in q$,
 
\item There is an $\Lang$-structure $\M$ and a tuple $(a_0, \dots, a_{n-1}) \in \M$ such that $\M \models \bigwedge\{\eta(a_{i_0}, \dots, a_{i_{k-1}}) \st \eta(x_{i_0}, \dots, x_{i_{k-1}}) \in q\}$. 

\end{itemize}

We say a partial quantifier free type has \defn{arity} $n$ if it is on a set of variables of size $n$. We denote the arity of a partial quantifier free type $q$ by $\arity(q)$.  We say a partial quantifier free type is a \defn{(complete) quantifier free type} if it is maximal under inclusion. 

\end{definition}

For $\M$ an $\Lang_\M$-structure and $\a \in \M$ we say $\a$ \defn{realizes} a (partial) quantifier free type $p(\x)$ if $\M \models \bigwedge_{\eta(\x) \in p(\x)} \eta(\a)$. We denote the collection of (complete) quantifier free types realized by elements of $\M$ by $\qftp(\M)$.

Throughout this paper we will be interested in constructing random structures in stages, first determining the structure of all singletons, then determining, based on the structure of the singletons, the structure of the pairs, etc. When doing this it is important that the complete structure of all $n$-tuples is determined before we determine the structure of the $(n+1)$-tuples. For this reason we will want to restrict our attention to the case where, whenever an $n$-ary relation holds of a tuple, the tuple has distinct elements (as otherwise the instance of the $n$-ary relation would be about a $k$-tuple of distinct elements for some $k < n$ and not about an $n$-tuple of distinct elements). To this end we define an important class of quantifier free types. 

\begin{definition}
Suppose $\eta(x_0, \dots, x_{n-1}) \in \Lwow(\Lang)$ is an atomic formula. We say $\eta(x_0, \dots, x_{n-1})$ is \defn{non-redundant} if for all $0 \leq i < j < n$ we have $x_i \neq x_j$. We say a partial quantifier free type is \defn{non-redundant} if every atomic formula in it, except perhaps those of the form $x_i = x_i$, are non-redundant. 

For $\x = (x_0, \dots, x_{n-1})$ distinct elements let $\nqftp_\Lang(\x)$ be the collection of non-redundant quantifier free types on $\x$ in $\Lang$. We will omit $\Lang$ when it is clear from context. 

Note $\nqftp_\Lang(\x)$ has a natural topology generated by clopen sets of the from $\{q \st \neg^{\ell} R(x_{i_0}, \dots, x_{i_{k-1}}) \in q\}$ where $R \in \Lang$, $\ell \in \{0, 1\}$ and $\{i_0, \dots, i_{k-1}\} \subseteq [n]$. With this topology $\nqftp_\Lang(\x)$ is homeomorphic to Cantor space. 

\end{definition}

\begin{definition}
Suppose we have a theory $T \in \Lwow(\Lang)$. We say a formula 
\[
\varphi(x_0, \dots, x_{n-1}) \in \Lwow(\Lang)
\]
is \defn{non-redundant} over $T$ if 
\[
T \models (\forall x_0, \dots, x_{n-1}) \left(\varphi(x_0, \dots, x_{n-1}) \rightarrow \bigwedge_{0 \leq i < j < n}  x_i \neq x_j\right).
\] 
We say $T$ has \defn{non-redundant quantifier free types} if each non-equality atomic formula with distinct variables is non-redundant over $T$. 

We say a structure $\M$ is \defn{non-redundant} if $\ScottSen{\M}$ has non-redundant quantifier free types.
\end{definition}
 
In particular if $T$ has non-redundant quantifier free types, then every quantifier free type realized in a model of $T$ is non-redundant. 

\begin{example}
Suppose $\M$ is a canonical structure. Then it is immediate from the definition of a canonical structure that $\ScottSen{\M}$ has non-redundant quantifier free types, and so $\M$ is non-redundant. 
\end{example}

Another important class of partial quantifier free types are those where the ordering of the variables in the formulas is consistent.

\begin{definition}
We say a partial quantifier free type $q$ on distinct elements $(x_0, \dots, x_{n-1})$ is \defn{ordered} if whenever $\eta(x_{i_0}, \dots, x_{i_{k-1}}) \in q$ and $0 \leq j^- < j^+ < k$ we have $i_{j^-} < i_{j^+}$. 

We define an \defn{ordered quantifier free type} to be a maximal ordered partial quantifier free type under inclusion (among the collection of ordered quantifier free types on the same variables). 

For $\x = (x_0, \dots, x_{n-1})$ let $\oqftp_\Lang(\x)$ be the collection of ordered quantifier free types in $\Lang$ on $\x$. Note $\oqftp_\Lang(\x)$ has a natural topology generated by clopen sets of the from $\{q \st \neg^{\ell} R(x_{i_0}, \dots, x_{i_{j-1}}) \in q\}$ where $R \in \Lang$,  $\ell \in \{0, 1\}$ and $0 \leq i_0 < \dots, < i_{j-1} < n$. This topology makes $\oqftp_\Lang(\x)$ homeomorphic to Cantor space. 

\end{definition}

An ordered quantifier free type need not itself be a (complete) quantifier free type as we can often find an extension which doesn't preserve the order of variables. Despite this we have chosen the current name as ``ordered partial quantifier free type'' does not convey the maximality condition which we need. 

The relationship between non-redundant and ordered quantifier free types is given by the following straightforward lemma. 
\begin{lemma}
\label{Non-redundant quantifier free types from ordered quantifier free types}
For any $q \in \nqftp_\Lang(\x)$ there is a unique collection $\<p_\tau\>_{\tau \in \Sym{\x}}$ such that 
\begin{itemize}
\item $p_\tau \in \oqftp_\Lang(\tau(\x))$,

\item $q = \bigcup_{\tau \in \Sym{\x}} p_\tau$.

\end{itemize}
\end{lemma}

The relationship between non-redundant and ordered quantifier free types plays an important role in translating between the distribution of an array of random variables which is $\Aut(\M)$-invariant and an $\Aut(\M)$-invariant measure on $\Str_{\Lang}(\M)$. 

In order to do this translation we want to consider arrays $\<f_p\>_{p \in \nqftp(\M)}$ where $f_p$ takes values in $\oqftp_\Lang(\x)$ where $|\x| = \arity(p)$ (which is itself a standard Borel space). In this way we will be able in \cref{Aut(M)-recipe} to use \cref{Non-redundant quantifier free types from ordered quantifier free types} to recover the (non-redundant) quantifier free type of a tuple $(n_0, \dots, n_{k-1})$ from the values of $\<f_{p_\tau}(\x)\>_{\tau \in \Sym{k}}$ when $\M \models p_\tau(\tau(n_0), \dots, \tau(n_{k-1}))$. We will then be able to recover an element of $\Str_\Lang(\M)$ from $\<f_p(\x)\>_{p \in \nqftp(\M)}$. 

We end this section recalling the notion of deduction in $\Lwow(\Lang)$ (see for example \cite{MR0424560} Sec. III.4). This will be important when discussing the theory of an ergodic invariant measure.

\begin{definition}
\label{Deduction rules in Lwow}
We say $\varphi \in \Lwow(\Lang)$ is a \defn{tautology} if for all $\Lang$-structures $\M$, we have $\M \models \varphi$. 

Suppose $T \subseteq \Lwow(\Lang)$ is a collection of sentences. We define the \defn{deductive closure} of $T$, $\dedcl(T)$, to be the smallest subset of $\Lwow(\Lang)$ such that 
\begin{itemize}
\item $T \subseteq \dedcl(T)$.

\item (Tautologies) $\dedcl(T)$ contains all tautologies. 

\item (Modus Ponens) If $\varphi \in \dedcl(T)$ and $(\varphi \rightarrow \psi) \in \dedcl(T)$ then $\psi \in \dedcl(T)$. 

\item (Generalization) If $(\forall v)(\varphi \rightarrow \psi(v)) \in T$ and $v$ is not free in $\varphi$, then $(\varphi \rightarrow (\forall v)\psi(v)) \in T$. 

\item (Conjunction) If $\bigwedge \Phi \in \Lwow(\Lang)$ and for all $\varphi \in \Phi$, $(\psi \rightarrow \varphi) \in \dedcl(T)$ then $(\psi \rightarrow \bigwedge \Phi) \in T$. 

\end{itemize}

We say $T$ is \defn{consistent} if $\dedcl(T) \neq \Lwow(\Lang)$. 

\end{definition}

While it is the case that if $T$ is countable and consistent it must have a model, this is not in general the case for uncountable $T$. However, if $A$ is a countable fragment then $T$ is an $A$ theory if and only if it is consistent and complete, i.e. for all sentences $\varphi \in A$ either $\varphi \in T$ or $\neg \varphi \in T$.

%%%%  %%%% %%%%  %%%%  %%%%  %%%%
\subsection{Definable Expansions}
\label{Definable Expansion Subsection}
%%%%  %%%% %%%%  %%%%  %%%%  %%%%
In this section we recall the notion of a definable expansion and show how they can be used to find $\Aut(\M)$-universal sentences with desired properties. This will then let us reduce the general problem of finding representations for $\Aut(\M)$-invariant measures on $\Str_{\Lang}(\M)$ to the problem of finding representations for such measures which concentrate on the collection of models which have our desired properties. 

\begin{definition}
Suppose $\Lang_0 \subseteq \Lang_1$, $T_0\in \Lwow(\Lang_0)$ and $T_1 \in \Lwow(\Lang_1)$. We say that $T_1$ is a \defn{definable expansion} of $T_0$ if 
\begin{itemize}
\item Every $\Lang_0$-structure $\M_0$ satisfying $T_0$ has a unique expansion to an $\Lang_1$-structure $\M_1$ satisfying $T_1$. 

\item $T_1 \models T_0$. 

\item For every formula $\varphi_1 \in \Lwow(\Lang_1)$ there is a formula $\varphi_0 \in \Lwow(\Lang_0)$ such that 
\[
T_1 \models (\forall \x)\ \varphi_1(\x) \leftrightarrow \varphi_0(\x).
\]

\end{itemize}
\end{definition}

So $T_1$ is a definable expansion of $T_0$ if all models of $T_1$ are also models of $T_0$, the restriction relation is a bijection, and further every formula in $\Lwow(\Lang_1)$ is equivalent (over $T_1$) to one in $\Lwow(\Lang_0)$.  The following lemma is then straightforward.

\begin{lemma}
\label{Definable expansions give isomorphic sets}
Suppose 
\begin{itemize}
\item $\Lang_0 \subseteq \Lang_1$, 

\item $T_0\in \Lwow(\Lang_0)$ and $T_1 \in \Lwow(\Lang_1)$, 

\item $T_1$ is a definable expansion of $T_0$.
\end{itemize}

Then for any $T^* \in \Lwow(\Lang_0)$,  $\extent{T_0 \And T^*}_\M$ is isomorphic to $\extent{T_1 \And T^*}_\M$ as $\Aut(\M)$-spaces. 
\end{lemma}

\cref{Definable expansions give isomorphic sets} tells us that if $T_1$ is a definable expansion of $T_0$ then when considered as $\Aut(\M)$-spaces, $\extent{T_0}_\M$ is isomorphic to $\extent{T_1}_{\M}$. We will in particular be interested in the case when two theories have a common definable expansion. 

\begin{definition}
We say theories $T_0 \in \Lwow(\Lang_0)$ and $T_1 \in \Lwow(\Lang_1)$ are \defn{interdefinable} when there is a language $\Lang_2 \supseteq \Lang_0 \cup \Lang_1$ and a theory $T_2$ which is a definable expansion of both $T_0$ and $T_1$.
\end{definition}

Two theories are interdefinable when it is possible to define each from the other (possibly in some larger language).

\begin{example}

Suppose $\M$ is an $\Lang$-structure. Then $\ScottSen{\M}$ and $\ScottSen{\CanStr{\Aut(\M)}}$ are interdefinable. 

\end{example}

Another important class of examples of interdefinable structuress are those obtained by simply relabeling the relations. 

\begin{definition}
A \defn{map of languages} between $\Lang_0$ and $\Lang_1$ is a function $i\: \Lang_0 \to \Lang_1$ such that for any relation $R\in \Lang_0$, $\arity(R) = \arity(i(R))$. A \defn{relabeling} of a language $\Lang_0$ by $\Lang_1$ is a bijective map of languages. 
\end{definition}

Note any relabeling extends to a bijection $i\: \Lwow(\Lang_0) \to \Lwow(\Lang_1)$. 

\begin{example}
If $i\:\Lang_0 \to \Lang_1$ is a relabeling then the empty theory in $\Lang_0$ is interdefinable with the empty theory in $\Lang_1$. 

\end{example}

There is an important example of a theory which is interdefinable with the empty theory in a language. 

For a language $\Lang$ let 
\[
\NonRedTh{\Lang} \defas \bigwedge_{\begin{subarray}{c}R \in \Lang \\ \arity(R) = k\end{subarray}} (\forall x_0, \dots, x_{k-1}) \left(R(x_0, \dots, x_{k-1}) \rightarrow \bigwedge_{0 \leq i < j < k}  x_i \neq x_j\right).
\]

For a language $\Lang$ let $\Lang_{\mathfrak{nr}} \defas \{R_{P, \equiv}\st P \in \Lang,$ of arity $n$,  $\equiv$ is an equivalence relation on $[n]$ with $\arity(R_{P, \equiv})$ many equivalence class$\}$ (here $\mathfrak{nr}$ stands for \qu{non-redundant}).

Let $\Th_{\Lang}^{\mathfrak{nr}}$ be the conjunction of all sentences of the form
\begin{align*}
(\forall x_0, \dots, x_{n-1})\  R_{P, \equiv}(x_{i_0}, \dots, x_{i_{k-1}}) &\leftrightarrow \bigg[P(x_0, \dots, x_{n-1}) \And \bigwedge_{0 \leq j < \ell < k} x_{i_j} \neq x_{i_\ell} \\
& \And \bigwedge_{\begin{subarray}{c}0 \leq j < \ell < n\\ j \equiv \ell\end{subarray}} x_{j} = x_{\ell} 
\And \bigwedge_{\begin{subarray}{c}0 \leq j < \ell < n\\ \neg(j \equiv \ell)\end{subarray}} x_{j} \neq x_{\ell}
\bigg]
\end{align*}

This theory replaces every relation with a sequence of non-redundant relations, one for every way in which arguments could be duplicate in the original relation.  

The following proposition is immediate from the definitions of $\NonRedTh{\Lang_{\mathfrak{nr}}}$ and $\Th_{\Lang}^{\mathfrak{nr}}$. 
\begin{proposition}
\label{non-redundant definable expansion}
If $\Lang$ is a language then 
\begin{itemize}
\item[(a)] $\w \cdot |\Lang_{\mathfrak{nr}}| = \w \cdot |\Lang|$. 

\item[(b)] $\Th_{\Lang}^{\mathfrak{nr}} \And \NonRedTh{\Lang_{\mathfrak{nr}}}$ is a definable expansion of the empty theory in $\Lang$ and $\NonRedTh{\Lang_{\mathfrak{nr}}}$ in $\Lang_{\mathfrak{nr}}$. 

\item[(c)] $\NonRedTh{\Lang_{\mathfrak{nr}}}$ has non-redundant quantifier free types. 
\end{itemize}
\end{proposition}

\cref{non-redundant definable expansion} (b) says that $\NonRedTh{\Lang_{\mathfrak{nr}}}$ is interdefinable with the empty theory in $\Lang$. And, as $\Lang$ has unbounded arity if and only if $\Lang_{\mathfrak{nr}}$ does, $\NonRedTh{\Lang_{\mathfrak{nr}}}$ is $\Aut(\M)$-universal precisely when $\Str_\Lang(\M)$ is a universal $\Aut(\M)$-space. We will prefer in most circumstances to work with $\extent{\NonRedTh{\Lang_{\mathfrak{nr}}}}_\M$ rather than with $\Str_\Lang(\M)$ and as such we define $\StrNR_\Lang(\M) \defas \extent{\NonRedTh{\Lang_{\mathfrak{nr}}}}_\M$, and $\StrNR_\Lang \defas \extent{\NonRedTh{\Lang_{\mathfrak{nr}}}}_{\M_\emptyset}$ where $\M_\emptyset$ is the unique structure on $\Nats$ in the empty language.

We now give a definable expansion which will gives us Morleyizations for a fragment. 

Given a countable fragment $A$ we let $\Lang_A \defas \{R_{\varphi(\x)}(\x) \st \varphi(\x) \in A\}$.  We define the sentence $\Th_A^{qe} \in \Lwow(\Lang_A)$ to be the conjunction of the following:
\begin{itemize}

\item If $\varphi(\x) = \neg \psi(\x)$ then $(\forall \x)R_{\varphi(\x)}(\x) \leftrightarrow \neg R_{\psi(\x)}(\x)$. 

\item If $\varphi(\x) = \bigwedge_{i\in I}\psi_i(\x)$ then $(\forall \x)R_{\varphi(\x)}(\x) \leftrightarrow \bigwedge_{i \in I}R_{\psi_i(\x)}(\x)$. 

\item If $\varphi(\x) = \bigvee_{i\in I}\psi_i(\x)$ then $(\forall \x)R_{\varphi(\x)}(\x) \leftrightarrow \bigvee_{i \in I}R_{\psi_i(\x)}(\x)$. 

\item If $\varphi(\x) = (\exists y)\psi(\x, y)$ then $(\forall \x)R_{\varphi(\x)}(\x) \leftrightarrow (\exists y)R_{\psi(\x, y)}(\x, y)$. 

\item If $\varphi(\x) = (\forall y)\psi(\x, y)$ then $(\forall \x)R_{\varphi(\x)}(\x) \leftrightarrow (\forall y)R_{\psi(\x, y)}(\x, y)$. 

\end{itemize}

Let $\Th_A^* \defas \bigwedge_{P \in \Lang}(\forall \x)P(\x) \leftrightarrow R_{P(\x)}(\x)$.  We call $\Th_A^{qe}\And \Th_A^*$ the \defn{Morleyization} of $A$. 

The following is immediate. 
\begin{proposition}
\label{Morleyization}
If $\Lang$ is a language and $A$ is a countable fragment of $\Lwow(\Lang)$ then 
\begin{itemize}
\item[(a)] $|\Lang_A| = \w$. 

\item[(b)] $\Th_A^{qe} \And \Th_A^*$ is a definable expansion of the empty theory in $\Lang$ and $\Th_A^{qe}$ in $\Lang_A$. Hence $\Th_A$ is interdefinable with the empty theory in $\Lang$, and if $\Lang$ has unbounded arity $\Th_A$ is $\Aut(\M)$-universal. 

\item[(c)] $\Th_A^{qe}\And \Th_A^*$ has Morleyization for $A$.

\end{itemize}

\end{proposition}

It is worth noting that in general $\Th_A$ will not have non-redundant quantifier free types. However, if we wish to obtain a universal $\Aut(\M)$-theory which both is Morleyized for $A$ and has non-redundant quantifier free types, we can first apply the above to get the theory $\Th_A$ and then apply the transformation to get an interdefinable non-redundant theory. This will result in a universal $\Aut(\M)$-theory which is Morleyized for $A$ and which also has non-redundant quantifier free types.

%%%%  %%%% %%%%  %%%%  %%%%  %%%%
\subsection{Invariant Measures}
\label{Invariant Measures Subsection}
%%%%  %%%% %%%%  %%%%  %%%%  %%%%

We now introduce the main objects of study in this paper, $\Aut(\M)$-invariant probability measures. In this subsection $G$ will be a Polish group and $(\circ, X)$ will be a $G$-space. 

\begin{definition}
\label{G-invariant definition}
Suppose $\mu$ is a measure on $X$. We say $\mu$ is \defn{$G$-invariant} if for all Borel sets $B \subseteq X$ and all $g \in G$
\[
\mu(B) = \mu(gB).
\]

\end{definition}

An important class of invariant measures are the ergodic ones. 

\begin{definition}
\label{Almost sure invariance and ergodic}
Suppose $\mu\in \ProbMeasures(X)$. We say a Borel subset $B \subseteq X$ is \defn{$\mu$-\as\ $G$-invariant} if for every $g \in G$, $\mu(B \SymDiff g^{-1}B) = 0$. 
We say $\mu$ is \defn{ergodic} if for every $\mu$-\as\ $G$-invariant Borel set $B$, either $\mu(B) = 0$ or $\mu(B) = 1$. 
\end{definition}

One of the reasons why ergodic $G$-invariant measures are important is that they the are also the extreme ones in the simplex of $G$-invariant measures.

The following is standard (see for example \cite{MR2161313} Lem.~A1.2).
\begin{lemma}
\label{Ergodic = extreme}
For a $G$-invariant measure $\mu$ on $X$ the following are equivalent
\begin{itemize}
\item $\mu$ is ergodic. 

\item $\mu$ is extreme, i.e.\ is not a non-trivial convex combination of $G$-invariant measures. 
\end{itemize}
\end{lemma}

The property of being extreme is important because of the following lemma.

\begin{lemma}[\cite{MR2161313} Thm.~A1.3]
\label{All measures are mixtures of extreme ones}
Every $G$-invariant measure is a mixture of extreme $G$-invariant measures. 
\end{lemma}

In particular this means that every $G$-invariant measure is a mixture of ergodic $G$-invariant measures. 

From the model theoretic point of view one of the most important consequences of ergodicity is that to each ergodic $\Aut(\M)$-invariant measure on $\Str_\Lang(\M)$ we can associate a complete consistent $\Lwow(\Lang_\M \cup \Lang)$-theory.

\begin{definition}
Suppose $\mu\in \ProbMeasures(\Str_\Lang(\M))$. Define the \defn{almost sure theory of $\mu$} to be
\[
\Th(\mu) \defas \{\tau \in \Lwow(\Lang_\M \cup \Lang)\st \mu(\extent{\tau}_\M) = 1\}.
\] 
\end{definition}

\begin{lemma}
\label{Th(mu) is consistent}
For any measure $\mu\in \ProbMeasures(\Str_\Lang(\M))$, $\Th(\mu)$ is consistent. 
\end{lemma}
\begin{Proof}
By $\sigma$-additivity of the measure $\mu$ we have $\Th(\mu)$ must be closed under the rules of deduction of $\Lwow(\Lang)$ in \cref{Deduction rules in Lwow}, i.e.\ we must have $\Th(\mu) = \dedcl(\Th(\mu))$. However $\mu(\extent{(\exists x)\ x \neq x}) =\mu(\emptyset) = 0 \neq 1$ and so $\Th(\mu)$ is consistent. 
\end{Proof}

For our purposes we are most interested in the theory of a measure when the measure is ergodic. 

\begin{lemma}
\label{Th(mu) is complete for ergodic mu}
If $\mu\in \ProbMeasures(\Str_\Lang(\M))$ is ergodic and $\Aut(\M)$-invariant then $\Th(\mu)$ is complete and consistent. 
\end{lemma}  
\begin{Proof}
For any sentence $\tau \in \Lwow(\Lang)$ we have that $\extent{\tau}_\M$ and $\extent{\neg \tau}_\M$ are invariant and hence $\mu(\extent{\tau}_\M) \in \{0, 1\}$ and $\mu(\extent{\neg\tau}_\M) \in \{0, 1\}$. Therefore one of $\tau$ or $\neg \tau$ is in $\Th(\mu)$ and so $\Th(\mu)$ is complete. 

The consistency of $\Th(\mu)$ follows from \cref{Th(mu) is consistent}.
\end{Proof}

We will end this section with a simple but important criteria for when a function can be extended to an $\Aut(\M)$-invariant measure on $\Str_\Lang(\M)$. 

\begin{definition}
For any language $\Lang$ let $\qfpi(\Lang)$ be the collection of formulas which are finite conjunctions of atomic and negations of atomic formulas with parameters in $\Nats$. 
\end{definition}

\begin{lemma}
\label{Measures from pi system}
Suppose $\mu^-\: \qfpi(\Lang_\M \cup \Lang) \to [0,1]$ is such that 
\begin{itemize}

\item[(a)] For every $\zeta \in \qfpi(\Lang_\M)$, if $\M \models \eta$ then $\mu^-(\eta) = 1$. 

\item[(b)] For every $\zeta \in \qfpi(\Lang_\M \cup \Lang)$ and every atomic $\Lang_\M \cup \Lang$-formula $\eta$ with parameters from $\Nats$,  
\[
\mu^-(\zeta) = \mu^-(\zeta \And \eta) + \mu^-(\zeta \And \neg \eta).
\]
\end{itemize}
Then there is a unique probability measure $\mu$ on $\Str_\Lang(\M)$ such that $\mu(\extent{\eta}_\M) = \mu^-(\eta)$ for all $\eta \in \qfpi(\Lang_\M \cup \Lang)$.  

Further if $\mu^-(\zeta(\a)) = \mu^-(\zeta(\b))$ whenever $\a, \b$ are in the same $\Aut(\M)$-orbit then $\mu$ is $\Aut(\M)$-invariant. 
\end{lemma}
\begin{Proof}
This follows immediately from the Carath\'{e}odory extension theorem.
\end{Proof}

Finally the following notion will be important.

\begin{definition}
Suppose $\Lang_0 \subseteq \Lang_1$ and $\mu \in \ProbMeasures(\Str_{\Lang_1}(\M))$. We let $\mu|_{\Lang_0}$ be the measure on $\Str_{\Lang_0}(\M)$ which agrees with $\mu$. Note $\mu|_{\Lang_0}$ is $\Aut(\M)$-invariant if $\mu$ is. 
\end{definition}

%%%%  %%%% %%%%  %%%%  %%%%  %%%%
\subsection{Aldous-Hoover-Kallenberg}
\label{Aldous-Hoover-Kallenberg Subsection}
%%%%  %%%% %%%%  %%%%  %%%%  %%%%

In this section we recall the Aldous-Hoover-Kallenberg theorem which gives a representation for $\Sym{\Nats}$-invariant measures. When $\M$ is free, a notion we will make precise in \cref{Free Structure}, we will be able to combine an $\Aut(\M)$-invariant measure $\mu$ with an explicit $\Sym{\Nats}$-invariant measure $\vartheta_\M$ concentrated on $\extent{\ScottSen{\M}}$ to get an $\Sym{\Nats}$-invariant measure. We will then use the Aldous-Hoover-Kallenberg theorem to get a representation of this combined measure from which we will be able to extract a representation of $\mu$. 

One of the definitive sources for the Aldous-Hoover-Kallenberg theorem and related facts is Chapter 7 of \cite{MR2161313} which we will use as a reference in this section. However there are two edge cases, \cref{Ergodic Sym(N)-Invariant Measures} and \cref{Ergodic Aldous-Hoover-Kallenberg Equivalence of Representations}, that we will need which, while not technically stated as results in \cite{MR2161313}, follow immediately with minor modifications of various proofs in \cite{MR2161313}. Our proofs of these facts will mention which results in \cite{MR2161313} are need to be modified and how, but will leave it to the enthusiastic reader to make the routine modifications.

The following theorem is (a variant of) what is often referred to as the Aldous-Hoover-Kallenberg theorem. See for example \cite{MR2161313} Ch.~7.5 or \cite{MR2426176}. Recall that a $U[0,1]$-array is a collection of uniform i.i.d.\ $[0,1]$-valued random variables.

\begin{theorem}[Aldous-Hoover-Kallenberg Theorem]
\label{Aldous-Hoover-Kallenberg Theorem}
Let $\X = (X_\a)_{\a \in \NatsIndex}$ be a  collection of random variables such that $X_\a$ takes values in a Polish space $S_{|\a|}$. Then the following are equivalent.
\begin{itemize}

\item For all $\tau \in \Sym{\Nats}$, $(X_\a)_{\a \in \NatsIndex} \eqd (X_{\tau(\a)})_{\a \in \NatsIndex}$, i.e.\ $(X_\a)_{\a \in \NatsIndex}$ is \defn{exchangeable}.

\item There exists a $U[0,1]$-array $(\zeta_\a)_{\a \in \Powerset_{< \w}(\Nats)}$ and a collection of measurable functions $f_n\:[0,1]^{\Powerset(n)} \to S_n$ such that 
\[
(X_\a)_{\a \in \NatsIndex} \eqd (f_{|\a|}(\widehat{\zeta}_\a))_{\a \in \NatsIndex}\ \as
\]
We call $\<f_n\>_{n \in \Nats}$ a \defn{representation} of $(X_\a)_{\a \in \NatsIndex}$

\end{itemize}

\end{theorem}

Often in the Aldous-Hoover-Kallenberg theorem as stated it is assumed that the array takes values in a single $S$, possibly $[0,1]$. However, as all standard Borel spaces of the same cardinality are isomorphic, \cref{Aldous-Hoover-Kallenberg Theorem} is no more general. By allowing $X_\a$ to take values in $\oqftp(\Lang^{|\a|})$ for some language $\Lang$, we can assume the array $(X_{\a})_{\a \in \NatsIndex}$ collectively takes values in $\Str_\Lang$. This then gives us the following equivalent formulation of \cref{Aldous-Hoover-Kallenberg Theorem}. Note the term \emph{recipe} is due to Austin (see \cite{MR2426176}).

\begin{definition}
\label{Sym(N)-recipe}
Suppose $\Lang$ is a language where all relations have arity less than $N \leq \w$. Define an \defn{$\Sym{\Nats}$-recipe} to consist of a collection of measurable functions $\f \defas \<f_n\>_{n < N}$ where 
\[
f_n\:[0,1]^{\Powerset(n)} \to \oqftp_{\Lang^n}(x_0, \dots, x_{n-1})
\] for $n < N$. 

We let $\RandMod{\f}\:[0,1]^{\Powerset_{<N}(\Nats)} \to \StrNR_\Lang$ be the function such that 
\begin{itemize}
\item for all relation $R \in \Lang$ of arity $k$, 

\item for all $\y = \<y_\b\>_{\b \in \Powerset_{<N}(\Nats)} \in [0,1]^{\Powerset_{<N}(\Nats)}$, and

\item for all $\a = (a_0, \dots, a_{k-1}) \in \Nats^{[<N]}$ with $k < N$
\end{itemize}
\[
\RandMod{\f}(\y) \models R(a_0, \dots, a_{k-1}) \text{ if and only if }R(x_0, \dots, x_{k-1}) \in f_{|\a|}(\widehat{y}_\a).
\]
We define the \defn{distribution} of $\f$ to be the distribution of $\RandMod{\f}$ (where the domain of $\RandMod{\f}$ is given the Lebesgue measure). 
\end{definition}

\begin{proposition}
Suppose $\mu$ is a measure on $\StrNR_{\Lang}$. Then the following are equivalent.
\begin{itemize}
\item $\mu$ is $\Sym{\Nats}$-invariant.

\item There exists an $\Sym{\Nats}$-recipe $\f$ whose distribution is $\mu$. 

\end{itemize}
\end{proposition}
\begin{Proof}
This follows immediately from  \cref{Aldous-Hoover-Kallenberg Theorem} by letting the space $S_n$ be $\oqftp_{\Lang^n}(x_0, \dots, x_{n-1})$. 
\end{Proof}

Also often considered part of the Aldous-Hoover-Kallenberg theorem is a characterization of when two exchangeable arrays have the same distribution. Before we can make this precise we need a definition. 

\begin{definition}
Suppose $g\: [0,1]^{\Powerset(n)} \to [0,1]$. We say $g$ \defn{preserves $\lambda$ in the highest order arguments} if for each $\<x_{\a}\>_{\a \in \Powerset(n) \setminus \{\{0, \dots, n-1\}\}}$ the map $x_{\{0, \dots, n-1\}} \to g(\<x_\a\>_{\a \in \Powerset(n)})$ preserves $\lambda$. 

Similarly, suppose $h\: [0,1]^{\Powerset(n)} \times [0,1]^{\Powerset(n)} \to [0,1]$. We say $h$ \defn{maps $\lambda^2$ to $\lambda$ in the highest order arguments} if for each $\<x_{\a}, y_\a\>_{\a \in \Powerset(n) \setminus \{\{0, \dots, n-1\}\}}$ the map $(x_{\{0, \dots, n-1\}}, y_{\{0, \dots, n-1\}}) \to h(\<x_\a\>_{\a \in \Powerset(n)}, \<y_\a\>_{\a \in \Powerset(n)})$ maps $\lambda^2$ to $\lambda$. 

\end{definition}

\begin{theorem}[\cite{MR2161313} Thm.~7.28]
\label{Aldous-Hoover-Kallenberg Equivalence of Representations}
Suppose $N \leq \w$, $\f^0 = \<f_n^0\>_{n < N}$ and $\f^1 = \<f_n^1\>_{n < N}$ are $\Sym{\Nats}$-recipes and  $(\zeta_\a)_{\a \in \Powerset_{<N}(\Nats)}$, $(\eta_\a)_{\a \in \Powerset_{<N}(\Nats)}$ are $U[0,1]$-arrays. Then the following are equivalent.
\begin{itemize}
\item $(f^0_{|\a|}(\widehat{\zeta}_\a))_{\a\in \Nats^{[<N]}} \eqd (f^1_{|\a|}(\widehat{\zeta}_\a))_{\a\in \Nats^{[<N]}}$. 

\item For each $n < N$ there are functions $g^0_n, g^1_n\: [0,1]^{\Powerset(n)} \to [0,1]$ which preserve $\lambda$ in the highest order arguments and are such that 
\[
(f^0_{|\a|}(\widehat{G^0_{\a}}(\widehat{\zeta}_\a)))_{\a\in \Nats^{[<N]}} = (f^1_{|\a|}(\widehat{G^1_{\a}}(\widehat{\zeta}_\a)))_{\a\in \Nats^{[<N]}}\ \as
\]
where $(G^i_{\a})_{\a\in \Nats^{[<N]}} = (g^i_{|\a|})_{\a\in \Nats^{[<N]}}$. 

\item For each $n < N$ there is a function $h_n\: [0,1]^{\Powerset(n)} \times [0,1]^{\Powerset(n)} \to [0,1]$ which maps $\lambda^2$ to $\lambda$ in the highest order arguments and is such that 
\[	
(f^0_{|\a|}(\widehat{\zeta}_\a))_{\a\in \Nats^{[<N]}} = (f^1_{|\a|}(\widehat{H_{\a}}(\widehat{\zeta}_\a, \widehat{\eta}_\a)))_{\a\in \Nats^{[<N]}}\ \as
\]
where $(H_{\a})_{\a\in \Nats^{[<N]}} = (h_{|\a|})_{\a\in \Nats^{[<N]}}$. 
\end{itemize}

\end{theorem}

Along with the representation theorem for exchangeable arrays, there is also a characterization of when an exchangeable array has an ergodic distribution. This characterization is both in terms of the possible $\Sym{\Nats}$-recipes, as well as in terms of the property of being dissociated. 

\begin{definition}
Let $\X = (X_\a)_{\a \in \NatsIndex}$ be a collection of random variables. We say $\X$ is \defn{dissociated}, if for every disjoint pair of tuples $\a$ and $\b$ the random variables $\widehat{X}_\a$ and $\widehat{X}_\b$ are independent. We say a measure is dissociated if it is the distribution of a collection of dissociated random variables. 
\end{definition}

\begin{theorem}[\cite{MR2161313} Lem.~7.35]
\label{Ergodic Sym(N)-Invariant Measures for Bounded Arity}
Suppose $\mu$ is an $\Sym{\Nats}$-invariant measure on $\Str_\Lang$ where the arities of $\Lang$ are bounded by $N < \w$. Then the following are equivalent.
\begin{itemize}
\item $\mu$ is ergodic. 

\item $\mu$ is dissociated. 

\item There is an $\Sym{\Nats}$-recipe $\f =\<f_n\>_{n < N}$ such that 
\begin{itemize}
\item $\RandMod{\f}$ has distribution $\mu$

\item For all $n < N$, $f_n$ does not depend on the coordinate with $\emptyset$-index. 
\end{itemize}

\end{itemize}

\end{theorem}

The next corollary follows from \cref{Ergodic Sym(N)-Invariant Measures for Bounded Arity} as well as from an argument which is essentially the same as that used in its proof. 
\begin{corollary}
\label{Ergodic Sym(N)-Invariant Measures}
Suppose $\Lang$ has unbounded arity and $\mu$ is an $\Sym{\Nats}$-invariant measure on $\StrNR_\Lang$. Then the following are equivalent.
\begin{itemize}
\item[(1)] $\mu$ is ergodic. 

\item[(2)] $\mu$ is dissociated. 

\item[(3)] There is an $\Sym{\Nats}$-recipe $\f =\<f_n\>_{n \in \Nats}$ such that 
\begin{itemize}
\item $\RandMod{\f}$ has distribution $\mu$,

\item For all $n \in \Nats$, $f_n$ does not depend on the coordinate with $\emptyset$-index. 
\end{itemize}

\end{itemize}

\end{corollary}
\begin{Proof}
First we show the equivalence of (1) and (3). Suppose (1) holds. Then for each $n \in \Nats$ we have that $\mu|_{\Lang^n}$ is ergodic as well. By \cref{Ergodic Sym(N)-Invariant Measures for Bounded Arity} we have an $\Sym{\Nats}$-recipe $\f^N =\<f^N_n\>_{n < N}$ for each $\mu|_{\Lang^{N-1}}$ such that each $f^N_n$ does not depend on the variable indexed by $\emptyset$. 

In order to combine these $\Sym{\Nats}$-recipes into a single $\Sym{\Nats}$-recipe for $\mu$ which satisfies (3) we need to proceed as in Lemma 7.21 of \cite{MR2161313} but with the modifications suggested in the last paragraph of the proof of Theorem 7.22 on p. 328 of \cite{MR2161313} and beginning the induction at $k = 1$. 

Next assume (3) holds. Because $\Str_\Lang$ is a Radon space $\mu$ is inner regular. For every measurable set $I$ we can therefore find a measurable set $A_\epsilon$ such that $\mu(I \SymDiff A_\epsilon) < \epsilon$ and which only depends on $\{0, \dots, n-1\}$ (for some $n$). The proof then follows Lemma 7.35 (iii) $\Rightarrow$ (i) of  \cite{MR2161313}. 

Now we show the equivalence of (1) and (2). First note that the Borel $\sigma$-algebra on $\StrNR_{\Lang}$ is generated by sets of the form $\extent{\varphi(\a)}$ where $\varphi$ is quantifier-free and first order. Therefore, $\mu$ is dissociated if and only if $\mu|_{\Lang_0}$ is dissociated for every finite $\Lang_0 \subseteq \Lang$. 

Now suppose (1) does not hold. Then there must be some finite $\Lang_0 \subseteq \Lang$ such that $\mu|_{\Lang_0}$ is not ergodic hence, by \cref{Ergodic Sym(N)-Invariant Measures for Bounded Arity}, not dissociated. But this then implies $\mu$ is not dissociated and so (2) does not hold. 

Finally, suppose (1) does hold. Then (3) holds as well. But then for every finite $\Lang_0 \subseteq \Lang$ there is a $\Sym{\Nats}$-recipe representing $\mu|_{\Lang_0}$ which does not depend on the coordinates with $\emptyset$-index. Therefore, once again by \cref{Ergodic Sym(N)-Invariant Measures for Bounded Arity}, for every finite $\Lang_0 \subseteq \Lang$, $\mu|_{\Lang_0}$ is dissociated. But this then implies $\mu$ is dissociated and so (2) holds. 
\end{Proof}
This result motivates the next definition. 
\begin{definition}
We say an $\Sym{\Nats}$-recipe $\f =\<f_n\>_{n < N}$ (for $N \leq \w$) is \defn{ergodic} if it doesn't depend on the coordinate indexed by $\emptyset$.
\end{definition}

Note that there may be $\Sym{\Nats}$-recipes which are not ergodic but whose distributions are, as not all representations of an ergodic measure satisfy the conditions of \cref{Ergodic Sym(N)-Invariant Measures}. However the distribution of any ergodic $\Sym{\Nats}$-recipe is ergodic and by \cref{Ergodic Sym(N)-Invariant Measures} every ergodic $\Sym{\Nats}$-invariant measure is the distribution of some ergodic $\Sym{\Nats}$-recipe.

\begin{proposition}
\label{Ergodic Aldous-Hoover-Kallenberg Equivalence of Representations}
Suppose $N \leq \w$, $\f^0 = \<f_n^0\>_{n < N}$ and $\f^1 = \<f_n^1\>_{n < N}$ are ergodic $\Sym{\Nats}$-recipes and $(\zeta_\a)_{\a \in \Powerset_{<N}(\Nats)}$, $(\eta_\a)_{\a \in \Powerset_{<N}(\Nats)}$ are $U[0,1]$-arrays. Then the following are equivalent.
\begin{itemize}
\item $(f^0_{|\a|}(\widehat{\zeta}_\a))_{\a\in \Nats^{[<N]}} \eqd (f^1_{|\a|}(\widehat{\zeta}_\a))_{\a\in \Nats^{[<N]}}$. 

\item For each $n < N$ there are functions $g^0_n, g^1_n\: [0,1]^{\Powerset(n)} \to [0,1]$ which preserve $\lambda$ in the highest order arguments, which don't depend on the argument indexed by $\emptyset$ and are such that 
\[
(f^0_{|\a|}(\widehat{G^0_{\a}}(\widehat{\zeta}_\a)))_{\a\in \Nats^{[<N]}} = (f^1_{|\a|}(\widehat{G^1_{\a}}(\widehat{\zeta}_\a)))_{\a\in \Nats^{[<N]}}\ \as
\]
where $(G^i_{\a})_{\a\in \Nats^{[<N]}} = (g^i_{|\a|})_{\a\in \Nats^{[<N]}}$. 

\item For each $n < N$ there is a function $h_n\: [0,1]^{\Powerset(n)} \times [0,1]^{\Powerset(n)} \to [0,1]$ which maps $\lambda^2$ to $\lambda$ in the highest order arguments, which don't depend on the argument indexed by $\emptyset$ and are such that 
\[
(f^0_{|\a|}(\widehat{\zeta}_\a))_{\a\in \Nats^{[<N]}} = (f^1_{|\a|}(\widehat{H_{\a}}(\widehat{\zeta}_\a, \widehat{\eta}_\a)))_{\a\in \Nats^{[<N]}}\ \as
\]
where $(H_{\a})_{\a\in \Nats^{[<N]}} = (h_{|\a|})_{\a\in \Nats^{[<N]}}$. 
\end{itemize}

\end{proposition}
\begin{Proof}
This proof follows closely that of Theorem 7.28 of \cite{MR2161313}. 
\end{Proof}

%%%%  %%%% %%%%  %%%%  %%%%  %%%%
\subsection{Existence of an $\Sym{\Nats}$-Invariant Measure}
\label{Existence of S-Nat-Invariant Measures Subsection}
%%%%  %%%% %%%%  %%%%  %%%%  %%%%

The following is an important component of our classification in \cref{Complete classification over structures with trivial dcl} of those sentences of $\Lwow(\Lang)$ which admit an $\Aut(\M)$-invariant measures, i.e. those sentences $\tau$ for which there is an $\Aut(\M)$-invariant measure concentrated on $\extent{\tau}_\M$.

\begin{theorem}[\cite{AFP-Complete-Classification}]
	\label{AFP-Complete classification}
For a sentence $\tau \in \Lwow(\Lang)$ the following are equivalent
\begin{itemize}
\item There is an $\Sym{\Nats}$-invariant measure concentrated on $\extent{\tau}$. 

\item There is an ergodic $\Sym{\Nats}$-invariant measure concentrated on $\extent{\tau}$. 

\item For all countable fragments $A \subseteq \Lwow(L)$ there is an $A$-theories $T$ with trivial dcl containing $\tau$. 

\item There is a countable fragment $A \subseteq \Lwow(L)$ and an $A$-theory $T$ with trivial dcl containing $\tau$. 

\end{itemize}
\end{theorem}

An immediate consequence of \cref{AFP-Complete classification} is the following.

\begin{theorem}[\cite{MR3515800} Thm.~1.1]
	\label{AFP-Main Result}
For an $\Lang$-structure $\M$ the following are equivalent
\begin{itemize}
\item There is an $\Sym{\Nats}$-invariant measure concentrated on $\extent{\ScottSen{\M}}$. 

\item $\M$ has trivial dcl.
\end{itemize}
\end{theorem}

%%%%  %%%%  %%%%  %%%% %%%%  %%%%  %%%%  %%%%
\section{Free Structures}
\label{Free Structures Section}
%%%%  %%%%  %%%%  %%%% %%%%  %%%%  %%%%  %%%%

In this section we introduce the abstract notion of a \emph{canonical structure} and we show that it corresponds, up to relabeling of the language, with being the canonical structure of a closed subgroup of $\Sym{\Nats}$. We then introduce the notion of a free canonical structure and show that every canonical structure has a minimal free extension (in a sense we make precise). 

%%%%  %%%% %%%%  %%%%  %%%%  %%%%
\subsection{Canonical Structures}
\label{Canonical Structures Subsection}
%%%%  %%%% %%%%  %%%%  %%%%  %%%%

\begin{definition}
We say a structure $\M$ is \defn{canonical} (for a language $\Lang_\M$) if
\begin{itemize}
\item $\M$ is ultrahomogeneous, 

\item $\M$ is non-redundant, 

\item for all relation $R \in \Lang_\M$ there is a tuple $\a \in M$ such that $\M \models R(\a)$, and 

\item for all $n \in \Nats$ and all $n$-tuples $\a \in \M$ with distinct entries there is a unique $n$-ary relation $R \in \Lang_\M$ such that $\M \models R(\a)$. 

\end{itemize}

\end{definition}

\begin{lemma}
The following are equivalent for an $\Lang_\M$-structure $\M$.
\begin{itemize}
\item $\M$ is canonical, 

\item there is a relabeling $i\: \Lang_\M \to \Lang_{\Aut(\M)}$ such that for any relation $R \in \Lang_\M$ and tuple $\n \in \Nats$, $\M \models R(\n)$ if and only if $\CanStr{\Aut(\M)} \models i(R)(\n)$, i.e. $\M$ and $\CanStr{\Aut(\M)}$ are the same structure up to a relabeling of the language.  

\end{itemize}

\end{lemma}
\begin{Proof} 
It is clear that for any group $G$, $\CanStr{G}$ is canonical. In the other direction if $\M$ is canonical, then for any relation $R \in \Lang$ and any tuple $\b$ such that $\M \models R(\b)$, we have $\{\a\st (\exists g \in \Aut(\M)) g\a = \b\} = \{\a\st \M \models R(\a)\}$ (as there is a unique relation holding of any tuple with distinct elements and $\M$ is ultrahomogeneous).
\end{Proof}

When $i_0, \dots, i_{k-1} < n$ we will use 
\[
R(x_0, \dots, x_{n-1})|_{(x_{i_0}, \dots, x_{i_{k-1}})} = P
\] 
as shorthand for the statement 
\[
(\forall x_0, \dots, x_{n-1})\ [R(x_0, \dots, x_{n-1}) \rightarrow P(x_{i_0}, \dots, x_{i_{k-1}})].
\]
If $\M$ is a canonical $\Lang$-structure, whenever 
\[
\M \models (\exists x_0, \dots, x_{n-1})\ [R(x_0, \dots, x_{n-1}) \And P(x_{i_0}, \dots, x_{i_{k-1}})]
\] 
we also have 
\[
\M \models R(x_0, \dots, x_{n-1})|_{(x_{i_0}, \dots, x_{i_{k-1}})} = P
\]
and hence for all $R(x_0, \dots, x_{n-1})$ and $i_0, \dots, i_{k-1} < n$ there is a unique $P \in\Lang$ such that $\M \models R(x_0, \dots, x_{n-1})|_{(x_{i_0}, \dots, x_{i_{k-1}})} = P$. We call $P$ the \defn{restriction} of $R(x_0, \dots, x_{n-1})$ to $(x_{i_0}, \dots, x_{i_{k-1}})$.

There is a natural notion of when one canonical structure is contained in another. 

\begin{definition}

If $\M$ is an ultrahomogeneous $\Lang$-structure let $\Age(\M)$ be the class of all finite $\Lang$-structures isomorphic to a finite substructure of $\M$. 
\end{definition}

\begin{definition}

Suppose $\M_0, \M_1$ are canonical structures in languages $\Lang_0, \Lang_1$ respectively. We say $\M_0 \subseteq_{can} \M_1$ if 
\begin{itemize}
\item[(i)] $\Lang_0 \subseteq \Lang_1$. 

\item[(ii)] For all $A_0 \in \Age(\M_0)$ there is an $A_1 \in \Age(\M_1)$ such that $A_0 = A_1|_{\Lang_0}$. In particular, as $\M_0$ and $\M_1$ are canonical such an $A_1$ is unique. 

\end{itemize} 
We say $\M_0 \sqsubseteq_{can} \M_1$ if there is a map of languages $i\:\Lang_0 \to \Lang_0^*$ such that $i(\M_0) \subseteq_{can} \M_1$. 

We say $\M_0 \preceq_{can} \M_1$ if there is a  relabeling $i\:\Lang_0 \to \Lang_0^*$ such that $i(\M_0) \subseteq_{can} \M_1$. 
\end{definition}

Note because our structures are canonical the $A_1$ in condition (ii) is unique and is trivial on $\Lang_1 \setminus \Lang_0$, i.e.\ $A_1 \models (\forall \x)\neg R(\x)$ for any $R \in \Lang_1 \setminus \Lang_0$.

An important fact about canonical structures is that if $\M_0 \subseteq_{can} \M_1$ then every $\Aut(\M_1)$-invariant measure on $\StrNR_{\Lang}(\M_1)$ restricts to an $\Aut(\M_0)$-invariant measure on $\StrNR_{\Lang}(\M_0)$. Specifically, suppose $\M_0 \subseteq_{can} \M_1$ and $\mu \in \StrNR_{\Lang}(\M_1)$. Let $\mu_0$ the map such that whenever 
\begin{itemize}
\item $P \in \Lang_{\M_0}$,

\item $\M_0 \models P(n_0, \dots, n_{k-1})$, 

\item $\M_1\models P(m_0, \dots, m_{k-1})$, and

\item $\eta \in \qfpi(\Lang)$ with parameters contained in $\{n_0, \dots, n_{k-1}\}$

\end{itemize}
then $\mu_0(\eta(n_0, \dots, n_{k-1})) = \mu(\eta(m_0, \dots, m_{k-1}))$. 

Note that as $\mu$ is $\Aut(\M_1)$ invariant, the value of $\mu_0$ is independent of the specific choice of $(m_0, \dots, m_{k-1})$, so long as the tuple satisfies $P$. 

By \cref{Measures from pi system} there is then a unique $\Aut(\M_0)$-invariant measure extending $\mu_0$, which we call the \defn{restriction} of $\mu$ to $\M_0$ and denote $\mu|_{\M_0}$.

%%%%  %%%% %%%%  %%%%  %%%%  %%%%
\subsection{Free Completions and Invariant Measures on Free Structures}
\label{Free Structures Subsection}
%%%%  %%%% %%%%  %%%%  %%%%  %%%%

For our purposes we will be interested in a very specific type of canonical structure.

\begin{definition}
\label{Free Structure}
Suppose $\M$ is a canonical structure. 
For $n \in \Nats$ let $X \defas \{x_k\}_{k \in [n+1]}$ be a set of distinct variables and and let $(\x_i)_{i \in [m]}$ be a collection of distinct sequences of elements of $X$ such that every element of $X$ is in at least one sequence. Further let $\x$ be an enumeration of $X$. 

We say a collection $\{R_i(\x_i)\}_{i \in [m]}$ is \defn{compatible} with $\M$ if
\begin{itemize}
\item whenever $\y$ is a sequence of distinct elements contained in $\x_i$ for some $i \in [m]$ there is a $j \in [m]$ such that $\x_j = \y$, and

\item for all $i, j \in [m]$, if $\y$ is a sequence of distinct elements contained in both $\x_i$ and $\x_j$ then $\M \models R_i(\x_i)|_{\y} = R_j(\x_j)|_{\y}$. 
\end{itemize}
We say $X$ is the collection of free variables of $\ol{R}$. 

We say an atomic formula $R^*(\x)$ of arty $n+1$ is an \defn{extension} of $\{R_i(\x_{A_i})\}_{i \in [m]}$ if for all $i \in [m]$, $\M \models R^*(\x) |_{\x_i} = R_i$. We say $\{R_i(\x_i)\}_{i \in [m]}$ is \defn{total} if it contains an extension of itself. 

We say that $\M$ is \defn{free} if all compatible collections have an extension. 

\end{definition}

If $\M$ is not canonical, we say $\M$ is free if $\CanStr{\Aut(\M)}$ is free.  A free structure can be thought of as a structure where any way of amalgamating types is consistent so long as it is locally consistent. 

\begin{example}
The quintessential example of a free structure is the Rado graph, \Rado. The canonical structure $\CanStr{\Aut(\Rado)}$ is the structure where for every finite graph $G$ there is a relation $R_G$ which holds exactly when the parameters form a graph isomorphic to $G$.
\end{example}

We also have the following example of a canonical structure which is not free. 
\begin{example}
Let $\Trado$ be the triangle free random graph. The canonical structure of $\Trado$ is the structure where for every finite triangle free graph $G$ there is a relation $R_G$ which holds exactly when the parameters form a graph isomorphic to $G$.

This structure is not free. To see this let $E$ be the graph with two elements and an edge between them. Then $\{E(x_0, x_1), E(x_1, x_2), E(x_2, x_0)\}$ is a compatible collection which is not the restriction of any relation in the canonical structure of $\Trado$. This is a quintessential example of how a canonical structure can fail to be free. 

\end{example}

Even though not all canonical structures are free, every canonical structure is contained in a free canonical structure. Further we can find a minimal such free extension.

\begin{lemma}
\label{Existence of free extension of canonical structure}
Suppose $\M$ is a canonical $\Lang_\M$-structure. Then there is a canonical structure $\FreeCS{\M}$ such that 
\begin{itemize}
\item $\FreeCS{\M}$ is free, 

\item $\M \subseteq_{can} \FreeCS{\M}$, and

\item whenever $\M \sqsubseteq_{can} \N$ and $\N$ is free then $\FreeCS{\M} \sqsubseteq_{can} \N$.

\end{itemize}

We call $\FreeCS{\M}$ the \defn{free completion} of $\M$.  

\end{lemma}
\begin{Proof}
For $n \in \w$ let $Y$ be the set of collections compatible with $\M$ which are not total. Let $\Lang_{\FreeCS{\M}} = \Lang_{\M} \cup \{P_{\ol{R}} \st \ol{R} \in Y\}$ where each $P_{\ol{R}}$ is a new relation in $\Lang_\M$ of the same arity as as $\ol{R}$. Let $T$ be the theory where
\begin{itemize}
\item $R(\x)|_\y = P \in T$ for all $R \in \Lang_\M$ and $P \in \Lang_\M$ such that $\M \models R(\x)|_{\y} = P$, and 
\end{itemize}
for $\ol{R} = \{R_i(\x_i)\}_{i \in [m]} \in Y$ of arity $[n]$ with $X$ the set free variables,
\begin{itemize}
\item $P_{\ol{R}}(\x) |_{\x_i} = R_i \in T$,  

\item $P_{\ol{R}}(\x) |_{\y} = P \in T$ when there is an $i \in [m]$ such that $\y \subseteq \x_i$ and $R_i(\x_i) |_\y = P \in T$, 

\item if $\sigma\:X \to X$ is a bijection and $\sigma\ol{R} = \{R_i(\sigma(\x_i))\}_{i \in [m]}$ then  $\ol{R}(\x)|_{\sigma(\x)} = \sigma\ol{R} \in T$, and

\item whenever 
\begin{itemize}
\item $Y \subseteq X$ and $\y$ is the subtuple of $\x$ containing the elements of $Y$, and

\item $\ol{S}(\y)$ is the collection of those $R_i(\x_i)$ such that all elements of $\x_i$ contained in $Y$, 

\end{itemize}
then $P_{\ol{R}}(\x) |_{\y} = P_{\ol{S}} \in T$.

\end{itemize}
For notational convenience if $\ol{R}$ is total of arity $[n]$ with $X$ the set free variables, $\x$ is an enumeration of $X$ and $R(\x) \in \ol{R}$ then we let $P_{\ol{R}} = R$. 

Let $\Age(\FreeCS{\M})$ be the collection of finite $\Lang_{\FreeCS{\M}}$-structures $A$ such that $A \models T$. Note it is easy to check that for each $A \in \Age(\FreeCS{\M})$ and tuple $\a \in A$ of distinct elements there is a unique relation $R\in \Lang_{\FreeCS{\M}}$ such that $A \models R(\a)$. 

For $A \in \Age(\FreeCS{\M})$ with enumeration $\a$ let $\ol{R}_\a  = \{R_i(\a_i)\}_{i \in [m]}$ be the collection of relations such that $R_i \in \Lang_\M$ and $\a_i \in A$. Note that $A \models P_{\ol{R}_\a}(\a)$. Further, if $B \in \Age(\FreeCS{\M})$ with enumeration $\b$ and $B \models P_{\ol{R}_\a}(\b)$ then the map $\a\mapsto \b$ is an isomorphism of $A$ and $B$.

\begin{claim}
\label{Age of FreeCS(M) is a Fraisse class}
$\Age(\FreeCS{\M})$ has the hereditary property (HP), the joint embedding property (JEP), and the disjoint amalgamation property (DAP). 
\end{claim}
\begin{Proof}<Subproof>
It is immediate that $\Age(\FreeCS{\M})$ has the HP and JEP. We now show it has the DAP. 
Suppose $X, Y \in \Age(\FreeCS{\M})$ are such that $X|_{X \cap Y} = Y|_{X \cap Y}$ and $\z$ is an enumeration of $X \cup Y$. Let $\ol{R}_{X, Y} = \{R_i(\z_i)\}_{i \in [m]}$ be the collection of relations such that $R_i \in \Lang_\M$ and either $\z_i \in X$ and $X \models R_i(\z_i)$,  or $\z_i \in Y$ and $Y \models R_i(\z_i)$. It is easy to check that $\ol{R}_{X, Y}$ is a collection compatible with $\M$. Further, if $Z\in \Age(\FreeCS{\M})$ is the structure with underlying set $X \cup Y$ such that $\ol{R}_\z = P_{\ol{R}_{X, Y}}$ then $Z$ is an amalgamation of $X$ and $Y$. 
\end{Proof}

As $\Age(\FreeCS{\M})$ satisfies (HP), (JEP), and (DAP) there is a unique (up to isomorphism) countable structure which is its \Fraisse\ limit. We denote this structure by $\FreeCS{\M}$. Note that by construction of $\Age(\FreeCS{\M})$ we have that $\FreeCS{\M}$ is free and canonical. 
All that is left is to show that $\FreeCS{\M}$ is $\preceq_{can}$-minimal among free structures containing $\M$. 

Suppose $\M \preceq_{can} \N$ with $\N$ free with the corresponding injection $i^*\: \Lang_\M \to \Lang_\N$. We will define our map $i\:\Lang_{\FreeCS{\M}} \to \Lang_\N$ by induction. Note $i$ will extend $i^*$. 

First, as $\Lang_{\M}^1 = \Lang_{\FreeCS{\M}}^1$ we let $i(U) = i^*(U)$ for all $U \in \Lang_{\FreeCS{\M}}^1$. 
Suppose $i$ has been defined for all relations of arity at most $n$ and let $Q$ be a relation of arity $n+1$. If $Q\in \Lang_{\M}$ let $i(Q) = i^*(Q)$. 

Suppose $Q \in \Lang_{\FreeCS{\M}} \setminus \Lang_\M$ and let $\ol{R} = \{R_i(\x_i)\}_{i \in [m]}$ be such $X$ is the set of free variables, $\x$ is an enumeration $X$ with $Q(\x) = P_{\ol{R}}(\x)$. For $y \in X$ let $\ol{R}^y$ be the collection of those $R_i(\x_i)$ with $\x_i \subseteq X \setminus \{y\}$ and let $\x_y$ be the corresponding enumeration of $X\setminus \{y\}$. Note that $\{\ol{R}^y(\x_y)\}_{y \in X}$ is a compatible collection with respect to $\FreeCS{\M}$ all of whose relations have arity $<n$. Therefore $\{i(P_{\ol{R}^y})(\x_y)\}_{y \in X}$ is a compatible collection with respect to $\N$. But $\N$ is free, and so there must be some extension $Q^+$ of $\{i(P_{\ol{R}^y})(\x_y)\}_{y \in X}$. Let $i(Q) = Q^+$. 
\end{Proof}

If $\M$ is not a canonical structure, then we define $\FreeCS{\M} \defas \FreeCS{\CanStr{\Aut(\M)}}$. It is worth noting that if $\M \subseteq_{can} \N$ then the corresponding map of languages that witnesses $\FreeCS{\M} \sqsubseteq_{can} \N$ is not in general unique. However for our purposes this will not be important. 

\begin{lemma}
If $\M$ is free then $\M$ has trivial dcl. 
\end{lemma}
\begin{Proof}
Note $\M$ has trivial dcl if and only if $\CanStr{\Aut(\M)}$, the canonical structure of $\M$ does. We can therefore assume without loss of generality that $\M$ is canonical and hence ultrahomogeneous. By \cite{MR1221741} Thm.~7.1.8, an ultrahomogeneous structure $\M$ in a relational language has trivial dcl if and only if $\Age(\M)$ has disjoint amalgamation. Now suppose $A, B \in \Age(\M)$ are such that $A|_{A\cap B} = B|_{A \cap B}$. Let $\c$ be an enumeration of $A \cup B$. Let $\ol{R}_A = \{R_i(\c_i)\st \c_i \subseteq A\text{ and }A \models R_i(\c_i)\}$ and $\ol{R}_B= \{R_i(\c_i)\st \c_i \subseteq B\text{ and }B \models R_i(\c_i)\}$. Then $\ol{R}_A \cup \ol{R}_B$ is a collection compatible with $\M$ and so there is an expansion $R^*$ of $\ol{R}_A \cup \ol{R}_B$. But there then must be some $C \in\Age(\M)$ such that $C \models R^*(\c)$, and any such $C$ is a strong amalgamation of $A$ and $B$. 
\end{Proof}

An important property of free canonical structures is that each comes with a natural $\Sym{\Nats}$-invariant measure concentrated on its isomorphism class.

\begin{definition}
\label{Uniform representation of free structures}
Suppose $\M$ is a free canonical structure with a chosen ordering on $\Lang_\M^n$ for each $n \in \Nats$. We define a \defn{uniform representation} on $\M$ to be an $\Sym{\Nats}$-recipe of the following form. \nl\nl
\ul{Arity $1$:} \nl
Let $L^1$ be the increasing enumeration of $\Lang_\M^1$ under the chosen ordering. Define $c^\M_1\:[0,1]^2 \to \oqftp_{\Lang_\M^1}(x_0)$ where $c^\M_1(y,z)$ is the type containing $\gamma_{L^1}(z)(x_0)$. Note the first coordinate is not used. This is representation we are defining is ergodic. 

Note that as $\M$ has trivial dcl and is canonical, for any unary relation $U$ there are infinitely many elements of $\M$ which satisfy $U$. Therefore $\RandMod{c^\M_1} \cong \M|_{\Lang_{\M}^1}$\ \as \nl\nl
\ul{Arity $n+1$:} \nl
Assume that $\RandMod{\<c^\M_i\>_{i \leq n}} \cong \M|_{\Lang_\M^{\le n}}$\ \as

Suppose $\ol{R} = \<R_i(\z_i)\>_{i \in [n]}$ is compatible with $\M$ with free variables $Z$ and where for all $Z^* \subsetneq Z$ there is some $i \in [n]$ such that $\x_i$ enumerates $Z^*$.  Then let $X_{\ol{R}}$ be the collection of extensions of $\ol{R}$ in $\M$ and assume $X_{\ol{R}}$ has the order induced by that of the language. Note $X_{\ol{R}}$ is non-empty as $\M$ is free. 

Given $\x \defas \<x_i\>_{i \in \Powerset(n+1)} \in [0, 1]^{\Powerset(n+1)}$ let $\x_i \defas \<x_i\>_{i \in \Powerset(n+1 \setminus\{i\})}$. Further let $\ol{R}^\x \defas \<c^\M_n(\widehat{\x}_i)\>_{i \in [n+1]}$. By induction $\ol{R}^\x$ is \as\ a compatible collection (as $\RandMod{\<c^\M_i\>_{i \leq n}} \cong \M|_{\Lang^{\le n}}$). 
Let $Y_{n+1}(\x) = X_{\ol{R}^\x}$, i.e.\ the collection of extensions of $\ol{R}^\x$.  Let $c^\M_{n+1}(\x) = \gamma_{Y_{n+1}(\x)}(x_{\{0, \dots, n\}})$. 
As $\M$ has trivial dcl it is easy to check that $\RandMod{\<c^\M_i\>_{i \leq n+1}} \cong \M|_{\Lang^{\leq n+1}}$\ \as\ \nl\nl
We let $\c^\M \defas \<c^\M_n\>_{n \in \Nats}$ and call it the \defn{\Erdos-\Renyi\ random $\M$-structure} in analog with the \Erdos-\Renyi\ random graph. 
\end{definition}

The random structure $\RandMod{\c^\M}$ can be thought of as first assigning to each element of $\Nats$ a unary relation in an i.i.d.\ manner. Then assigning a binary relation to every pair of elements in a manner which is i.i.d.\ conditioned on the unary types that were previously assigned. Then assigning to each triple a ternary relation in a manner which is i.i.d.\ conditioned on the binary types that were defined, etc. 

\begin{lemma}
\label{Erdos-Renyi random structure is ergodic}
The distribution of $\RandMod{\c^\M}$ is ergodic.
\end{lemma}
\begin{Proof}
This follows immediately from \cref{Ergodic Sym(N)-Invariant Measures} and the definition of $\c^\M$. 
\end{Proof}
\begin{lemma}
The distribution of $\RandMod{\c^\M}$ concentrates on $\extent{\ScottSen{\M}}$.
\end{lemma}
\begin{Proof}
First the quantifier free types realized by $\RandMod{\c^\M}$ are the same as those realized in $\M$ almost surely. Let $\vartheta_\M$ be the distribution of $\RandMod{\c^\M}$. 

Suppose $\M \models R(\x, y)|_{\x} = P$. By a back and forth argument it suffices to show that $\RandMod{\c^\M}$ satisfies $(\forall \x)P(\x) \rightarrow (\exists y) R(\x, y)\ \as\ $ or equivalently that $\vartheta_\M(\extent{(\forall \x)P(\x) \rightarrow (\exists y) R(\x, y)}) = 1$. 

But as $\vartheta_\M$ is countably additive and $\Sym{\Nats}$-invariant it therefore suffices to show $\vartheta_\M(\extent{P(0, \dots, {k-1}) \rightarrow (\exists y)R(0, \dots, {k-1}, y)}) = 1$ where $\arity(P) = k$. 

Once again using countable additivity and $\Sym{\Nats}$-invariance it suffices to show $\vartheta_\M(\extent{P(0, \dots, {k-1}) \And R(0, \dots, {k-1}, k)}) > 0$. But by construction we know not only that $\vartheta_\M(\extent{R(0, \dots, {k-1}, k) \rightarrow P(0, \dots, {k-1})}) = 1$ but also that $\vartheta_\M(\extent{R(0, \dots, {k-1}, k)}) > 0$, and so we are done. 
\end{Proof}

The representations $\c^\M$ will play an important role in showing that all $\Aut(\M)$-invariant measures, for $\M$ free, are representable. First though we will need a little more notation. 

\begin{definition}
\label{S_p and alpha_p}
For all $k \geq n$ and $p \in \oqftp_{\Lang_\M^{\leq n}}(x_0, \dots, x_{k-1})$ define 
\[
S_p \defas \{\x \in [0,1]^{\Powerset_{\leq n}(k)}\st (\exists \y) \RandMod{\c^\M}(\x,\y)\models p(0, \dots, k-1)\}
\] 
\ul{Arity $1$:}\nl
For each $U(x_0) \in \oqftp_{\Lang_\M^1}(x_0)$ let $S_U^\M = (c_1^\M)^{-1}(U(x_0))$. As $S_U$ is the product of $[0,1]$ with an interval, let $\alpha_U\: [0,1]\times [0,1] \to S_U$ be a homeomorphism which doesn't depend on the first argument such that for every set $X \subseteq S_U$, $\lambda(X)/\lambda(S_U) = \lambda(\alpha_U^{-1}(X))$. 

Now suppose $p = \bigwedge_{i \in [k]} U_i(x_i)$ is an element of $\oqftp_{\Lang_\M^1}(x_0, \dots, x_{k-1})$. We let $\alpha_p\:[0,1] \times [0, 1]^k \to S_p$ be the map where $\alpha_p(y, \<x_i\>_{i < k}) = \<\alpha_{U_i(y, x_i)}\>_{i < k}$.  \nl\nl 
\ul{Arity $n+1$:}\nl
Assume $k \geq n$ and for all $p \in \oqftp_{\Lang_\M^{\leq n}}(x_0, \dots, x_{k-1})$ we have defined the map $\alpha_p\:[0,1]^{\Powerset_{\leq n}(k)} \to S_p$ which is a bijection that doesn't depend on the coordinate indexed by $\emptyset$. 

Let $p \in \oqftp_{\Lang_\M^{\leq n+1}}(x_0, \dots, x_n)$ and let $p^-$ be the restriction of $p$ to the space $\oqftp_{\Lang_\M^{\leq n}}(x_0, \dots, x_n)$. Let $S_p = (c_{n+1}^\M)^{-1}(p(\x))$. Then because of how $\c^\M$ was defined we know that $S_p = S_{p^-} \times I$ where $I$ is a subinterval of $[0,1]$. Let $i\: [0,1] \to I$ be an isomorphism such that for every set $X \subseteq I$, $\lambda(X)/\lambda(I) = \lambda(i^{-1}(X))$. Let $\alpha_p = \alpha_{p^-} \times i$. 

We then extend the definition of $\alpha_p$ to the case where  $p \in \oqftp_{\Lang_\M^{\leq n+1}}(x_0, \dots, x_{k-1})$ and $k > n+1$ in the obvious way. 
\end{definition}

Note that for each $p \in \oqftp(\M)$ of arity $n$ we have $p$ only depends on $\Lang_\M^{\leq n}$ and so the following holds 
\begin{itemize}
\item If $X\subseteq S_p$ then $\lambda(X)/ \lambda(S_p) = \lambda(\alpha_p^{-1}(X))$. 

\item For any $n_0 < n$ if $q_{n_0}$ is the restriction of $p$ to $\Lang_\M^{\leq n_0}$ then $S_p = S_q \times I_{p, q}$ where $I_{p, q}$ is the product of intervals. Further $\alpha_p$ restricts to $\alpha_q$ on $S_q$.

\end{itemize}

%%%%  %%%%  %%%%  %%%% %%%%  %%%%  %%%%  %%%%
\section{Merging of Measures}
\label{Merging of Measures Section}
%%%%  %%%%  %%%%  %%%% %%%%  %%%%  %%%%  %%%%

One of the key observations of this paper is that if $\Aut(\M) \subseteq \Aut(\N)$ then we can combine any $\Aut(\N)$-invariant measure concentrated on $\extent{\ScottSen{\M}}_\N$ with any $\Aut(\M)$-invariant measure, in a unique way, to get an $\Aut(\N)$-invariant measure which agrees with both.  Further every $\Aut(\N)$-invariant measure concentrated on $\extent{\ScottSen{\M}}_\N$ is of this form. 

In particular when $\M$ is a structure with underlying set $\Nats$ and trivial dcl we will be able reduce the question of \qu{for which $\varphi \in \Lwow(\Lang)$ is there an $\Aut(\M)$-invariant measure on $\Str_{\Lang}(\M)$ which concentrates on $\extent{\varphi}_{\M}$?} to the question of \qu{for which $\varphi \in \Lwow(\Lang)$ is there an $\Sym{\Nats}$-invariant measure which concentrates on $\extent{\varphi \And \ScottSen{\M}}$?}, a question which has been completely answered by Ackerman, Freer and Patel in \cite{AFP-Complete-Classification}.

Further we will be able to use this method of combining measures to show that when $\M$ is free every $\Aut(\M)$-invariant measure can be combined with the distribution of the \Erdos-\Renyi\ random $\M$-structure to get an $\Sym{\Nats}$-invariant measure from whose representation we can extract a representation for the $\Aut(\M)$-invariant measure we started with.

Suppose 
\begin{itemize}
\item $\M$ and $\N$ are canonical structures with underlying set $\Nats$ and $\Aut(\M) \subseteq \Aut(\N)$. 

\item $\nu_\M$ is an $\Aut(\N)$-invariant measure on $\StrNR_{\Lang_\M}(\N)$ which is concentrated on $\extent{\ScottSen{\M}}$. 

\item $\mu$ is an $\Aut(\M)$-invariant measure on $\StrNR_\Lang(\M)$.

\end{itemize}

Let $\B(\M, \Lang)$ be the collection of sets of the form 
\[
\extent{P(n_0, \dots, n_{k-1}) \And \varphi(n_{i_0}, \dots, n_{i_{j-1}})}
\]
where
\begin{itemize}
\item $P \in \Lang_\M^k$, and

\item $\varphi(n_{i_0}, \dots, n_{i_{j-1}}) \in \qfpi(\Lang)$. 
\end{itemize}

Suppose 
\begin{itemize}
\item $P \in \Lang_\M^k$, $Q \in \Lang_\N^k$ and $\{\a \in \Nats \st \M \models P(\a)\} \subseteq \{ \a \in \Nats \st \N \models Q(\a)\}$, and

\item $\M \models P(n_0, \dots, n_{k-1})$ and $\N \models Q(n_0, \dots, n_{k-1})$.

\end{itemize}

The for $\extent{P(n_0, \dots, n_{k-1}) \And \varphi(n_{i_0}, \dots, n_{i_{j-1}})} \in \B(\M, \Lang)$ we let 
\begin{align*}
\mu \boxplus^\dagger &\nu_\M(\extent{P(n_0, \dots, n_{k-1}) \And \varphi(n_{i_0}, \dots, n_{i_{\ell-1}})}) \defas \\
&\nu_\M(\extent{P(n_0, \dots, n_{k-1})}_\N) \cdot \mu(\extent{\varphi(n_{i_0}, \dots, n_{i_{\ell-1}})}_\M). 
\end{align*}

\begin{proposition}
\label{Properties of mu boxplus nu_M}
$\mu \boxplus^\dagger \nu_\M$ extends uniquely to a $\Aut(\N)$-invariant measure, $\mu \boxplus \nu_\M$, on $\StrNR_{\Lang_\M \cup \Lang}(\N)$ such that 
\begin{itemize}
\item[(i)] $\mu \boxplus \nu_\M(\extent{P(\n)}_\N) = \nu_\M(\extent{P(\n)}_\N)$ for all $p \in \Lang_\M$ and $\n \in \Nats$.

\item[(ii)] $\mu \boxplus \nu_\M(\extent{\eta(\n)}_\N) = \mu(\extent{\eta(\n)}_\M)$ for all $\eta(\n) \in \qfpi(\Lang)$ and $\n \in \Nats$.
\end{itemize}
\end{proposition}
\begin{Proof}
Clearly any extension of $\mu \boxplus^\dagger \nu_\M$ to a $\Aut(\N)$-invariant measure will satisfy (i) and (ii) and each $\Aut(\N)$-invariant measure satisfying (i) and (ii) extends $\mu \boxplus^\dagger \nu_\M$. It therefore suffices to show that there is a unique such extension. We will do this by showing that there is a unique extension to a map $\mu \boxplus^- \nu_\M$ on $\qfpi(\Lang_\N \cup \Lang_\M \cup \Lang)$ which satisfies the conditions of \cref{Measures from pi system}. 

Suppose $\eta \in \qfpi(\Lang_\N \cup \Lang_\M \cup \Lang)$. Then $\eta = \zeta_0 \And \zeta_1\And \zeta_2$ for some $\zeta_0 \in \qfpi(\Lang_\M)$, $\zeta_1 \in \qfpi(\Lang)$ and $\zeta_2 \in \qfpi(\Lang_\N)$. 

If $\N \models \zeta_2$ let $\mu\boxplus^- \nu_\M(\eta) = \mu \boxplus^- \nu_\M(\zeta_0 \And \zeta_1)$. Otherwise let $\mu \boxplus^- \nu_\M(\eta) = 0$. In particular this implies that $\mu \boxplus^- \nu_\M$ satisfies \cref{Measures from pi system} condition (a) and condition (b) for atomic formulas from $\Lang_\N$. 

Now assume $\N \models \zeta_2$. Let $\a$ be the union of the parameters from $\zeta_1$ and $\zeta_0$ and let $q = R(\a)$ where $R \in \Lang_\N$ is the unique relation which holds of $\a$ (under some fixed ordering). Let $X_{\zeta_0} \defas \{P \in \Lang_\M \st (\exists\x\in \Nats) \N \models R(\x)$ and $\M \models P(\x) \And \zeta_0(\y),$ where $\y$ sits in $\x$ as the parameters of $\zeta_0$ sit in $\a\}$. If $X_{\zeta_0}$ is empty then let $\mu\boxplus^- \nu_\M(\eta) = 0$. 

If $X_{\zeta_0}\neq \emptyset$ then let $\mu\boxplus^- \nu_\M(\eta) = \sum_{P \in X_{\zeta_0}} \mu\boxplus^\dagger \nu_\M(\extent{P \And \zeta_1}_\N)$. 

Note that 
\[
\sum_{P \in X_{\zeta_0}} \mu\boxplus^\dagger \nu_\M(\extent{P \And \zeta_1}_\N) = \bigg(\sum_{P \in X_{\zeta_0}} \nu_\M(\extent{P}_\N)\bigg) \cdot \mu(\extent{\zeta_1}_\M)
\]
but as $\{\extent{P}\st P \in X_{\zeta_0}\}$ are disjoint this sum equals 
\[
\nu_\M\left(\Extent{\bigvee_{P \in X_{\zeta_0}}P}_\N\right) \cdot \mu(\extent{\zeta_1}_\M).
\]
In particular this implies the range of $\mu \boxplus^- \nu_\M$ is a subset of $[0, 1]$.

We now must show \cref{Measures from pi system} condition (b) is satisfied for atomic formulas from $\Lang_\M \cup \Lang$. 

Suppose $R \in \Lang_\M$. For $P \in X_{\zeta_0}$, if $\M \models P(\x)|_{\y} = R$ and $\b$ sits in $\a$ as $\y$ sits in $\x$ then 
\[
\mu \boxplus^- \nu_\M((P(\a) \And R(\b)) \And \zeta_1) = \mu\boxplus^- \nu_\M(P(\a) \And \zeta_1)
\]
and 
\[
\mu \boxplus^- \nu_\M((P(\a) \And \neg R(\b)) \And \zeta_1) = 0.
\]
Similarly if $\M \models P(\x)|_{\y} \neq R$ then 
\[
\mu \boxplus^- \nu_\M((P(\a) \And \neg R(\b)) \And \zeta_1) = \mu\boxplus^- \nu_\M(P(\a) \And \zeta_1)
\] 
and 
\[
\mu \boxplus^- \nu_\M((P(\a) \And R(\b)) \And \zeta_1) = 0.
\] 

Therefore 
\begin{align*}
\mu \boxplus^- \nu_\M(\zeta_0 &\And \zeta_1) =\sum_{P \in X_{\zeta_0}} \mu\boxplus^- \nu_\M(P \And \zeta_1) \\
&=
\left(\sum_{\begin{subarray}{c}P \in X_{\zeta_0}\\ p(\x)|_\y = R(\y)\end{subarray}} \mu\boxplus^- \nu_\M(P \And \zeta_1)\right)+ \left(\sum_{\begin{subarray}{c}P \in X_{\zeta_0}\\ P(\x)|_\y \neq R(\y)\end{subarray}} \mu\boxplus^- \nu_\M(P \And \zeta_1)\right) \\
&= 
\mu \boxplus^- \nu_\M((\zeta_0 \And R(\b)) \And \zeta_1) + \mu \boxplus^- \nu_\M((\zeta_0 \And \neg R(\b)) \And \zeta_1).
\end{align*}

Now suppose $\beta \in \qfpi(\Lang)$. We have 
\[
\mu \boxplus^- \nu_\M(\zeta_0 \And (\zeta_1 \And \beta)) = \nu_\M\left(\Extent{\bigvee_{P \in X_{\zeta_0}}P}\right) \cdot \mu(\extent{\zeta_1 \And \beta})
\]
and so if $\beta$ is an atomic formula
\[
\mu\boxplus^- \nu_\M(\zeta_0 \And \zeta_1) = \mu \boxplus^- \nu_\M(\zeta_0 \And (\zeta_1 \And \beta))  +\mu \boxplus^- \nu_\M(\zeta_0 \And (\zeta_1\And \neg \beta)).
\]
We therefore have, by  \cref{Measures from pi system}, that there is a unique measure, $\mu \boxplus \nu_\M$, extending $\mu \boxplus^- \nu_\M$. Further, as $\mu \boxplus^- \nu_\M$ is $\Aut(\M)$-invariant so is $\mu \boxplus \nu_\M$. 

Finally it is also immediate that $\mu \boxplus^- \nu_\M$ is the unique extension of $\mu \boxplus^\dagger \nu_\M$ which preserves additivity of the measure and hence satisfies the conditions of \cref{Measures from pi system}. 
\end{Proof}

The value $\mu\boxplus \nu_\M(\extent{\eta(\n)}_\N)$ has a particularly nice description when $\eta \in \Lwow(\Lang_\M\cup \Lang_\N \cup \Lang)$. For $\eta \in \Lwow(\Lang_\M \cup \Lang_\N \cup \Lang)$, $\n \in \Nats$, $q \in \qftp(\N)$ with $\N \models q(\n)$ and $p \in \qftp(\M)$ with $\{\a \st \M \models p(\a)\} \subseteq \{\a \st \N \models q(\a)\}$ define $\mu^p(\extent{\eta(\n)}_\N) = \mu(\extent{\eta(\n^*)}_\M)$  for some $\n^*$ where $\M \models p(\n^*)$. Note as $\mu$ is $\Aut(\M)$-invariant this is well defined. 

\begin{proposition}
\label{Description of mu boxplus nu_M}
We have
\[
\mu\boxplus \nu_\M(\extent{\eta(\n)}) = \sum_{\begin{subarray}{c}p \in \qftp(\M) \\ \{\a \st \M \models p(\a)\} \subseteq \{\a \st \N \models q(\a)\}\end{subarray}} \nu_\M(\extent{p(\n)}) \cdot \mu^p(\extent{\eta(\n)}_\N)
\]
\end{proposition}
\begin{Proof}
Let $A$ be a fragment containing $\eta$. As $\Th_A$ is interdefinable with the empty theory in $\Lang_\M$ there is a unique measure $\mu_A$ on $\extent{\Th_A}$ which agrees with $\mu$ on $\Lang$ and a unique measure on $\extent{\Th_A}$ which agrees with $\mu \boxplus \nu_\M$ on $\Lang \cup \Lang_\M$. But then $\mu_A \boxplus \nu_\M$ agrees with $\mu \boxplus \nu_\M$ on $\Lang \cup \Lang_\M$ and hence it must be that measure. But the proposition holds for $\mu_A \boxplus \nu_\M$ as $\Th_A$ proves $\eta$ is equivalent to a quantifier free formula. Therefore the proposition must also hold of $\mu \boxplus \nu_\M$. 
\end{Proof}

We have shown how to uniquely recover an $\Aut(\N)$-invariant measure from an $\Aut(\N)$-invariant measure concentrated on $\extent{\ScottSen{\M}}_\N$ along with an $\Aut(\M)$-invariant measure. We next show that every $\Aut(\N)$-invariant measure concentrated on $\extent{\ScottSen{\M}}_\N$ must have such a (necessarily unique) decomposition.

\begin{proposition}
\label{Decomposition of Aut(N) invariant measure on concentrated on M}
Suppose $\eta$ is an $\Aut(\N)$-invariant measure on $\StrNR_{\Lang_\M \cup \Lang}(\N)$ concentrated on $\extent{\ScottSen{\M}}_\N$. There exists unique measures $\mu^\eta$ and $\nu_\M^\eta$ such that 
\begin{itemize}
\item $\mu^\eta$ is an $\Aut(\M)$-invariant measure on $\StrNR_{\Lang}(\M)$, 

\item $\nu_\M^\eta$ is an $\Aut(\N)$-invariant measure on $\StrNR_{\Lang_\M}(\N)$ concentrated on $\extent{\ScottSen{\M}}_\N$, 

\item $\eta = \mu^\eta \boxplus \nu_\M^\eta$. 

\end{itemize}

\end{proposition}
\begin{Proof}
First notice that if $\mu^\eta$ and $\nu_\M^\eta$ exist they must agree with $\eta$ on their respective domains and hence are uniquely determined by $\eta$. 

By \cref{Measures from pi system}, to define $\mu^\eta$ it suffices to define it on $\qfpi(\Lang_\M \cup \Lang)$. In particular suppose $\zeta_0 \in \qfpi(\Lang_\M)$ and $\zeta_1 \in \qfpi(\Lang)$. We let $\mu^\eta(\extent{\zeta_0 \And \zeta_1}_\M) = \eta(\extent{\zeta_0 \rightarrow \zeta_1}_\N)$ if $\M \models \zeta_0$ and $0$ otherwise. It is then easily checked that the conditions of \cref{Measures from pi system} are satisfied so that this partial definition of $\mu^\eta$ extends uniquely to an $\Aut(\M)$-invariant measure. 

Similarly by \cref{Measures from pi system} to define $\nu_\M$ it suffices to define it on $\qfpi(\Lang_\N \cup \Lang_\M)$. In particular suppose $\zeta_0 \in \qfpi(\Lang_\M)$ and $\zeta_1 \in \qfpi(\Lang_\N)$. We let $\nu_\M^\eta(\extent{\zeta_0 \And \zeta_1}_\N) = \eta(\extent{\zeta_0}_\N)$ if $\N \models \zeta_1$ and $0$ otherwise. It is then easily checked that the conditions of \cref{Measures from pi system} are satisfied so that this partial definition of $\nu_\M^\eta$ extends uniquely to an $\Aut(\N)$-invariant measure. 

But we have that $\eta$ and $\mu \boxplus \nu_\M$ agree on the domains of $\mu$ and $\nu_\M$ and so by \cref{Properties of mu boxplus nu_M} we have $\eta = \mu \boxplus \nu_\M$. 
\end{Proof}

%%%%  %%%% %%%%  %%%%  %%%%  %%%%
\subsection{Inherited Properties Of Merged Measures} 
\label{Inherited Properties Of Merged Measures Subsection}
%%%%  %%%% %%%%  %%%%  %%%%  %%%%

We now put together the results obtained in the previous subsection to get several key results about merged measures which are inherited from their components.

\begin{proposition}
\label{Equivalence of ergodicity of mu and mu boxplus nu_M}
Suppose $\Aut(\M) \subseteq \Aut(\N)$,  $\nu_\M$ is an $\Aut(\N)$-invariant measure concentrated on $\extent{\ScottSen{\M}}_\N$ and $\mu$ is an $\Aut(\M)$-invariant measure $\mu$ on $\StrNR_{\Lang}(\M)$. Then the following are equivalent.
\begin{itemize}
\item[(1)] $\nu_\M$ is an ergodic $\Aut(\N)$-invariant measure and $\mu$ is an ergodic $\Aut(\M)$-invariant measure. 

\item[(2)] $\mu \boxplus \nu_\M$ is an ergodic $\Aut(\N)$-invariant measure. 
\end{itemize}
\end{proposition}
\begin{Proof}
First we show (2) implies (1). Suppose $\mu \boxplus \nu_\M$ is ergodic. Note that if $\alpha_0, \alpha_1 \in (0, 1)$, $\nu_\M^0, \nu_\M^1$ are $\Aut(\N)$-invariant measures and $\nu_\M = \alpha_0 \cdot \nu_\M^0 + \alpha_1\cdot \nu_\M^1$ then $\mu \boxplus \nu_\M = \alpha_0 \cdot \mu \boxplus \nu_\M^0 + \alpha_1 \cdot \mu \boxplus \nu_\M^1$ by \cref{Description of mu boxplus nu_M}. Further, by \cref{Decomposition of Aut(N) invariant measure on concentrated on M} we have $\mu \boxplus \nu_\M = \mu \boxplus \nu^i_\M$ if and only if $\nu_\M = \nu_\M^i$ (for $i \in \{0,1\}$). Therefore as $\mu \boxplus \nu_\M$ is ergodic (by assumption) so is $\nu_\M$. 

Similarly if $\mu^0, \mu^1$ are $\Aut(\M)$-invariant measures and $\mu = \alpha_0 \cdot \mu^0 + \alpha_1\cdot \mu^1$ then $\mu \boxplus \nu_\M = \alpha_0 \cdot \mu^0 \boxplus \nu_\M + \alpha_1 \cdot \mu^1 \boxplus \nu_\M$ by \cref{Description of mu boxplus nu_M}. Also by \cref{Decomposition of Aut(N) invariant measure on concentrated on M} we have $\mu \boxplus \nu_\M = \mu^i \boxplus \nu_\M$ if and only if $\mu = \mu^i$ (for $i \in \{0,1\}$). So $\mu$ is extreme (and hence ergodic) as well. 

Next we show (1) implies (2). Suppose $\mu, \nu_\M$ are ergodic. Further suppose $\eta_0, \eta_1$ are $\Aut(\N)$-invariant measures on $\StrNR_\Lang(\N)$ and $\mu \boxplus \nu_\M = \alpha_0 \cdot \eta_0 + \alpha_1 \cdot \eta_1$ with $\alpha_0, \alpha_1 \in (0, 1)$. 
 
By \cref{Decomposition of Aut(N) invariant measure on concentrated on M} there are unique $\Aut(\M)$-invariant measures $\nu_\M^{\eta_0}, \nu_\M^{\eta_1}$ concentrated on $\extent{\ScottSen{\M}}_\N$ and $\Aut(\N)$-invariant measures $\mu^{\eta_0}, \mu^{\eta_1}$ such that $\eta_i = \mu^{\eta_i} \boxplus \nu_\M^{\eta_i}$ (for $i \in \{0, 1\}$). 

But, for any $p \in \Lang_\M$ and $\n \in\Nats$ by \cref{Properties of mu boxplus nu_M} we have 
\begin{align*}
\nu_\M(\extent{p(\n)}) &= \mu \boxplus \nu_\M(\extent{p(\n)}) \\
&= \alpha_0 \cdot \mu^{\eta_0} \boxplus \nu_\M^{\eta_0}(\extent{p(\n)}) + \alpha_1 \cdot \mu^{\eta_1} \boxplus \nu_\M^{\eta_1}(\extent{p(\n)}) \\ 
&= \alpha_0 \cdot \nu_\M^{\eta_0}(\extent{p(\n)}) + \alpha_1 \cdot \nu_\M^{\eta_1}(\extent{p(\n)}).
\end{align*}
Therefore we have $\nu_\M = \alpha_0 \cdot \nu_\M^{\eta_0} + \alpha_1 \cdot \nu_\M^{\eta_1}$ and so, as $\nu_\M$ is ergodic we have $\nu_\M = \nu_\M^{\eta_0} = \nu_\M^{\eta_1}$. 

But then we also have
\[
\mu \boxplus \nu_\M = \alpha_0 \cdot \mu^{\eta_0} \boxplus \nu_\M + \alpha_1 \cdot \mu^{\eta_1} \boxplus \nu_\M = (\alpha_0 \cdot \mu^{\eta_0} + \alpha_1 \cdot \mu^{\eta_1}) \boxplus \nu_\M
\]
with the second equality following from \cref{Description of mu boxplus nu_M}. Finally, this implies by \cref{Decomposition of Aut(N) invariant measure on concentrated on M} that $\mu = \alpha_0 \cdot \mu^{\eta_0} + \alpha_1 \cdot \mu^{\eta_1}$. So, as $\mu$ is ergodic we have $\mu = \mu^{\eta_0} = \mu^{\eta_1}$. 

We therefore have $\mu \boxplus \nu_\M = \eta_0 = \eta_1$ and so $\mu \boxplus \nu_\M$ is extreme and hence by \cref{Ergodic = extreme}, $\mu \boxplus \nu_\M$ is ergodic. We therefore have (1) implies (2). 
\end{Proof}

The following is an immediate consequence of \cref{Properties of mu boxplus nu_M}.

\begin{proposition}
\label{Equivalence of theories of mu and mu boxplus nu_M}
Suppose $\Aut(\M) \subseteq \Aut(\N)$,  $\nu_\M$ is an $\Aut(\N)$-invariant measure concentrated on $\extent{\ScottSen{\M}}_\N$ and $\mu$ is an $\Aut(\M)$-invariant measure $\mu$ on $\StrNR_{\Lang}(\M)$. Then $\Th(\mu) = \Th(\mu \boxplus \nu_\M)$. 
\end{proposition}
\begin{Proof}
For $\tau$ a sentence in $\Lwow(\Lang_\M \cup \Lang)$ let $\tau^=(x) \defas \tau \And (x = x)$ be a unary formula. Note $\tau \in \Th(\mu)$ if and only if $\mu(\extent{\tau^=(0)}) = 1$. Let $q \in \qftp(\N)$ be such that $\N \models q(0)$. 

By \cref{Description of mu boxplus nu_M} we have 
\[
\mu\boxplus \nu_\M(\extent{\tau^=(0)}) = \sum_{\begin{subarray}{c}p \in \qftp(\M) \\ \{\a \st \M \models p(\a)\} \subseteq \{\a \st \N \models q(\a)\}\\
\end{subarray}} \nu_\M(\extent{p(0)}) \cdot \mu^p(\extent{\tau^=(0)}_\N).
\]
But 
\[
\sum_{\begin{subarray}{c}p \in \qftp(\M) \\ \{\a \st \M \models p(\a)\} \subseteq \{\a \st \N \models q(\a)\}\\
\end{subarray}} \nu_\M(\extent{p(0)}) =1
\]
by construction and so 
\[
\mu\boxplus \nu_\M(\extent{\tau^=(0)}) = 1
\]
if and only if 
\[
\bigwedge_{\begin{subarray}{c}p \in \qftp(\M) \\ \{\a \st \M \models p(\a)\} \subseteq \{\a \st \N \models q(\a)\}\\
\end{subarray}} [\mu^p(\extent{\tau^=(0)}_\N) = 1].
\]
Hence $\tau \in \Th(\mu)$ if and only if $\tau \in \Th(\mu \boxplus \nu_\M)$. 
\end{Proof}

In particular we have the following easy consequence of \cref{AFP-Complete classification} and \cref{Equivalence of theories of mu and mu boxplus nu_M}. 

\begin{proposition}
\label{Complete classification over structures with trivial dcl}
Suppose $\M$ has trivial dcl and $\varphi \in \Lwow(\Lang)$. Then the following are equivalent.
\begin{itemize}
\item[(a)] There is an $\Aut(\M)$-invariant measure on $\Str_\Lang(\M)$ which concentrates on $\extent{\varphi}_\M$. 

\item[(b)] There is an ergodic $\Aut(\M)$-invariant measure on $\Str_\Lang(\M)$ which concentrates on $\extent{\varphi}_\M$. 

\item[(c)] For all countable fragment $A \subseteq \Lwow(\Lang \cup \Lang_\M)$ there is an $A$-theory $T$ with trivial dcl which contains $\ScottSen{\M} \And \varphi$. 

\item[(d)] There is a countable fragment $A \subseteq \Lwow(\Lang \cup \Lang_\M)$ and an $A$-theory $T$ with trivial dcl which contains $\ScottSen{\M} \And \varphi$. 
\end{itemize}

\end{proposition}
\begin{Proof}
First notice that (b) immediately implies (a) and that (c) and (d) are equivalent by \cref{AFP-Complete classification}. 

Now suppose (a) holds. Note that the linear mixture of measures concentrated on $\extent{\neg \varphi}_\M$ is itself concentrated on $\extent{\neg \varphi}_\M$. Therefore, by \cref{Ergodic = extreme} and \cref{All measures are mixtures of extreme ones} there must exist an ergodic $\Aut(\M)$-invariant measure $\mu$ concentrated on $\extent{\varphi}_\M$. Therefore (b) holds. 

Further, as $\M$ has trivial dcl, by \cref{AFP-Main Result} there is an ergodic measure $\nu_\M$ concentrated on $\extent{\ScottSen{\M}}$.  Now suppose (b) holds and $\mu$ is an ergodic $\Aut(\M)$-invariant measure concentrated on $\extent{\varphi}_\M$. By \cref{Properties of mu boxplus nu_M} $\mu \boxplus \nu_\M$ is an $\Sym{\Nats}$-invariant measure, by \cref{Equivalence of ergodicity of mu and mu boxplus nu_M} $\mu \boxplus \nu_\M$ is ergodic, and by \cref{Equivalence of theories of mu and mu boxplus nu_M} $\mu \boxplus \nu_\M$ concentrates on $\extent{\ScottSen{\M} \And \varphi}$. Now let $A$ be any countable fragment containing $\ScottSen{\M} \And \varphi$. As $\mu \boxplus \nu_\M$ is ergodic we have $\Th(\mu \boxplus \nu_\M) \cap A$ is an $A$-theory and so, by \cref{AFP-Complete classification} has trivial dcl. Therefore (c) holds.

Suppose (c) holds. Let $A$ be a countable fragment and $T$ an $A$-theory which as trivial dcl and contains $\ScottSen{\M} \And \varphi$.  Then there is an $\Sym{\Nats}$-invariant measure $\beta$ concentrated on $\extent{\ScottSen{\M}\And \varphi}$ by \cite{AFP-Complete classification}. Further, by \cref{Decomposition of Aut(N) invariant measure on concentrated on M} there is an $\Aut(\M)$-invariant measure $\mu$ such that $\beta = \mu \boxplus \nu_\M$ for some $\Sym{\Nats}$-invariant measure $\nu_\M$ concentrated on $\extent{\ScottSen{\M}}$. But by \cref{Equivalence of theories of mu and mu boxplus nu_M} we have $\Th(\beta) = \Th(\mu)$ and so $\mu$ must concentrate on $\extent{\varphi}_\M$ and so (a) holds. 
\end{Proof}

Note that the case special case of \cref{Complete classification over structures with trivial dcl} where  $\varphi$ is the Scott sentence of a structure $\N$ follows immediately from Theorem~3.21 and Theorem~4.1 in \cite{MR3515800}.

%%%%  %%%%  %%%%  %%%% %%%%  %%%%  %%%%  %%%%
\section{Representations}
\label{Representations Section}
%%%%  %%%%  %%%%  %%%% %%%%  %%%%  %%%%  %%%%

We now finally prove \cref{Aut(M)-invariant measures for M free are representable} and \cref{Equivalence of extension to free structure and representability}, the main results of this paper, which gives a characterization of those $\Aut(\M)$-invariant measures which are representable. We then also give several properties of such representable measures.

%%%%  %%%% %%%%  %%%%  %%%%  %%%%
\subsection{$\Aut(\M)$-Recipes}
\label{Aut(M)-Recipes Subsection}
%%%%  %%%% %%%%  %%%%  %%%%  %%%%

We now introduce $\Aut(\M)$-recipes, which are an immediate generalizations of $\Sym{\Nats}$-recipes from \cref{Sym(N)-recipe}. In the following definition we assume $\M$ is canonical. 
\begin{definition}
\label{Aut(M)-recipe}
An \defn{$\Aut(\M)$-recipe} on $\Lang$ is a collection of measurable functions $\f = \<f_p\>_{p \in \nqftp(\M)}$ where $f_p\: [0,1]^{\Powerset(n)} \to \oqftp_{\Lang^n}(x_0, \dots, x_{n-1})$ when $\arity(p) = n$.

If $\f = \<f_p\>_{p \in \nqftp(\M)}$ is an $\Aut(\M)$-recipe let $\RandMod{\f}\: [0,1]^{\Powerset_{<\w}(\Nats)} \to \Str_{\Lang}(\M)$ be the function such that
\begin{itemize}
\item for any relation $R \in \Lang$ of arity $k$, 

\item $\y = \<y_\b\>_{\b \in \Powerset_{<\w}(\Nats)} \in [0,1]^{\Powerset_{<\w}(\Nats)}$, and

\item $\a = (a_0, \dots, a_{k-1}) \in \NatsIndex$ with $\M \models p_\a(\a)$ for $p_\a \in \nqftp(\M)$. 
\end{itemize}
\[
\RandMod{\f}(\y) \models R(a_0, \dots, a_{k-1}) \text{ if and only if }R(x_0, \dots, x_{k-1}) \in f_{p_\a}(\widehat{y}_\a)
\]
We define the \defn{distribution} of $\f$, $\mu_\f$, to be the distribution of $\RandMod{\f}$ where the domain of $\RandMod{\f}$ is given the Lebesgue measure. 
\end{definition}

An $\Aut(\M)$-recipe is similar to an $\Sym{\Nats}$-recipe except in the construction of the random structure we are allowed to use the type of the elements in $\M$. In particular every $\Aut(\M)$-recipe gives rise to an $\Aut(\M)$-invariant measure.

\begin{lemma}
Suppose $\f = \<f_p\>_{p \in \nqftp(\M)}$ is an $\Aut(\M)$-recipe. Then $\mu_\f$ is $\Aut(\M)$-invariant.
\end{lemma}
\begin{Proof}
This follows from the fact that if $\a, \b \in \M$ with $\a$ and $\b$ satisfying the same quantifier free type, $p$, and $\zeta \in \qfpi(\Lang)$, then $\mu_\f(\extent{\zeta(\a)}_\M) = \lambda(f^{-1}_{p}(\zeta(\x))) = \mu_\f(\extent{\zeta(\b)}_\M)$. 
\end{Proof}

We will be interested in distributions of $\Aut(\M)$-recipes.

\begin{definition}
We say an $\Aut(\M)$-invariant measure $\mu$ is \defn{representable} if there is an $\Aut(\M)$-recipe $\f$ such that $\mu = \mu_\f$. In this case we say $\f$ is a \defn{representation} of $\mu$. 

\end{definition}

The following is important as it shows every representable $\Aut(\M_0)$-invariant measure can be extended (in a not necessarily unique way) to an $\Aut(\M_1)$-invariant measure whenever $\M_0 \subseteq_{can} \M_1$.

\begin{proposition}
\label{Representable measures can be extended}
Suppose $\f$ is an $\Aut(\M_0)$-recipe on $\Lang$ and $\M_0 \subseteq_{can} \M_1$. Then there is an $\Aut(\M_1)$-recipe on $\g$ on $\Lang$ such that whenever $p_0 \in \nqftp(\M_0)$, $p_1 \in \nqftp(\M_1)$ with $p_0 \subseteq p_1$ we have $g_{p_1} = f_{p_0}$.  
\end{proposition}
\begin{Proof}
For $p \in \nqftp(\M_1)$, if $p|_{\Lang_{\M_1}} = q \in \nqftp(\M_0)$ let $g_p = f_q$. Otherwise let $g_p$ be the constant function which takes the value of the trivial type, i.e.\ the type $\{\neg R(\x)\st R \in \Lang^{\arity(p)}\}$.
\end{Proof}

In particular \cref{Existence of free extension of canonical structure} and \cref{Representable measures can be extended} imply that any representable measure is the restriction of a measure which is invariant under $\Aut(\M)$ where $\M$ is free. We will next see that the converse holds as well and every measure which is $\Aut(\M)$-invariant for $\M$ free has a representation.

%%%%  %%%% %%%%  %%%%  %%%%  %%%%
\subsection{Representability and Free Structures}
\label{Representability and Free Structures Subsections}
%%%%  %%%% %%%%  %%%%  %%%%  %%%%

We now show that whenever $\M$ is free, every $\Aut(\M)$-invariant measure is representable. We can assume with out loss of generality that $\M$ is canonical. 

Recall $\c^\M$ from \cref{Uniform representation of free structures} is an $\Sym{\Nats}$-recipe such that $\RandMod{\c^\M}$ is concentrated on $\extent{\ScottSen{\M}}$. Call its distribution $\vartheta_\M$. Also recall the definitions of $S_p$ and $\alpha_p$ (for $p \in \oqftp(\M)$) from \cref{S_p and alpha_p}. In what follows $\c^\M$, $S_p$ and $\alpha_p$ (for $p \in \oqftp(\M)$) will play an important role. 

We begin with a technical lemma.

\begin{lemma}
\label{Sym(N)-recipe agreeing with c^M}
Suppose $\M$ is free and $\mu$ is an $\Aut(\M)$-invariant measure. Then there is an $\Sym{\Nats}$-recipe $\e = \<e_n\>_{n \in \Nats}$ whose distribution is $\mu \boxplus \vartheta_\M$ and which agrees with $\c^\M$. Further if $\mu$ is ergodic $\e$ can be chosen to be ergodic as well. 

\end{lemma}
\begin{Proof}
By \cref{Properties of mu boxplus nu_M} we know that $\mu \boxplus \vartheta_\M$ is an $\Sym{\Nats}$-invariant measure and so there is an $\Sym{\Nats}$-recipe $\g$ which represents it. Let $\g^*$ be the restriction of $\g$ to the language $\Lang_\M$. As $\mu \boxplus \vartheta_\M$ restricted to $\Lang_\M$ has the same distribution as $\vartheta_\M$ we have that $\RandMod{\g^*}$ has the same distribution as $\RandMod{\c^\M}$, i.e.\ $\vartheta_\M$. So, by \cref{Aldous-Hoover-Kallenberg Equivalence of Representations}, there are maps $h^\c_n, h^\g_n\:[0,1]^{\Powerset(n)} \to [0,1]$ such that
\begin{itemize}
\item $h^\c_n, h^\g_n$ preserve $\lambda$ in the highest order arguments, and 

\item $c^\M_n \circ \widehat{h^\c_n} = g^*_n \circ \widehat{h^{\g}_n}$\ \as 
\end{itemize}
Further, by \cref{Ergodic Aldous-Hoover-Kallenberg Equivalence of Representations} and the fact that $\c^\M$ is ergodic, when $\mu$ is also ergodic we can assume the maps $h^\c_n$ and $h^\g_n$ are independent of the coordinate indexed by $\emptyset$. 

But, by \cref{Aldous-Hoover-Kallenberg Equivalence of Representations} $\<g_n \circ \widehat{h^\g_n}\>_{n\in \Nats}$ is also a representation of $\mu \boxplus \vartheta_\M$. We can therefore assume without loss of generality that $g^*_n = c^\M_n \circ \widehat{h^\c_n}$\ \as\  

We now define $e_n$ by induction. \nl\nl
\ul{Arity $1$:}\nl
For each $U \in \nqftp(\M)$ of arity $1$ let $P_U = (h^\c_1)^{-1}(S_U)$. By construction of $\vartheta_\M$ we know that $P_U = [0, 1] \times I_U$ where $|I_U| = 2^\w$ and is Borel with $\lambda(I_U) = \lambda(S_U)$. So there is a measure preserving isomorphism $\beta_U\: S_U \to I_U$. Let $e_1\:S_U \to \nqftp_\Lang(x_0)$ be such that $e_1(x,a) = g_1 \circ \beta_U(a)$ whenever $a \in S_U$. Clearly $\RandMod{g_1}$ and $\RandMod{e_1}$ have the same distribution and $e_1$ does not depend on the first coordinate when $g_1$ doesn't.  \nl\nl
\ul{Arity $n+1$:}\nl
Assume we have defined $e_n$ and further that  $P_p = (h^\c_n)^{-1}(S_p)$ for any $p \in \nqftp_{\Lang^{\leq n}}(x_0, \dots, x_{\arity(p)-1})$. Further suppose we have defined a measure preserving bijection $\beta_p\: S_p \to P_p$ such that $e_n = g_n \circ \beta_p$ on $S_p$. 

Now suppose $p \in \nqftp_{\Lang_\M^{n+1}}(x_0, \dots, x_{n+1})$ and $p^-$ is the restriction of $p$ to $n$-ary types. Let $X_{p^-} = \{p^* \in \qftp(\M)\st p^*$ restricts to $p^-\}$. 

Suppose $\x \in [0,1]^{\Powerset(n) \setminus \{\{0, \dots, n-1\}\}}$ and $y \in [0,1]$.
As $g_{n+1}(\x, y)\cap \Lang_\M = c^\M_{n+1} \circ (\widehat{h^\c_{n+1}})(\x, y)$ we must have 
\begin{align*}
(\x, y) \in \bigcup_{p^* \in X_{p^-}}P_{p^*}&\Leftrightarrow (\widehat{h^\c_{n+1}})(\x, y) \in \bigcup_{p^* \in X_{p^-}}S_p\\
&\Leftrightarrow (\widehat{h^\c_{n}})(\x) \in S_{p^-}\\
&\Leftrightarrow \x \in P_{p^-}.
\end{align*}
 Therefore $\bigcup_{p^* \in X_{p^-}} P_{p^*} = P_{p^-} \times [0,1]$. 

But we also have that \as\ for all $\x \in P_{p^-}$ and all $p \in X_{p^-}$ that $\lambda(\{y\st (\x, y) \in P_p\}) = \lambda(\{y \st \gamma_{X_{p^-}}(y) = p\})$ and hence does not depend on $\x$. For each $p \in X_{p^-}$ there is therefore a measure preserving isomorphism $i_p\: P_{p^-} \times \{y \st \gamma_{X_{p^-}}(y) = p\} \to P_p$. 

But we know by induction that there is a measure preserving isomorphism $\beta_p\:S_{p^-} \to P_{p^-}$ such that $e_{\arity(p^-)} =  g_{\arity(p^-)} \circ \beta_{p^-}$. So if we let $\beta_p = \beta_{p^-} \times i_p$ then $\beta_p$ is a measure preserving isomorphism from $S_p$ to $P_p$. We then let $e_{n+1}(\x) = g_{n+1} \circ \beta_p(\x)$ whenever $\x \in S_p$. 

It is then immediate that $\e$  agrees with $\c^\M$ and that $\e$ is ergodic when $\mu$ is. 
\end{Proof}

\begin{theorem}
\label{Aut(M)-invariant measures for M free are representable}
Suppose $\M$ is free and $\mu$ is an $\Aut(\M)$-invariant measure. Then $\mu$ is representable. 

\end{theorem}
\begin{Proof}
By \cref{Sym(N)-recipe agreeing with c^M} there is an $\Sym{\Nats}$-recipe $\e$ with distribution $\mu\boxplus \vartheta_\M$ and which agrees with $\c^\M$. But by construction, for every $p \in \qftp(\M)$, $\alpha_p$ is a bijection from $[0,1]^{\Powerset(\arity(p))}$ to $S_p$. Therefore $\f \defas \<e_p \circ \alpha_p\>_{p \in \qftp(\M)}$ is an $\Aut(\M)$-recipe which is a representation of $\mu$. 
\end{Proof}

\begin{theorem}
\label{Equivalence of extension to free structure and representability}
The following are equivalent for an $\Aut(\M)$-invariant measure $\mu$.
\begin{itemize}
\item[(a)] $\mu$ is representable. 

\item[(b)] There is an $\Aut(\FreeCS{\M})$-invariant measure $\FreeCS{\mu}$ such that $\mu = \FreeCS{\mu}|_{\M}$. 
\end{itemize}

\end{theorem}
\begin{Proof}
This follows immediately from \cref{Representable measures can be extended} and \cref{Aut(M)-invariant measures for M free are representable}. 
\end{Proof}

Condition (b) from \cref{Equivalence of extension to free structure and representability} can be thought of as being an amalgamation condition on the measure, i.e.\ it says that any locally consistent properties of the measure can be amalgamated into a measure where they are globally consistent.

%%%%  %%%% %%%%  %%%%  %%%%  %%%%
\subsection{Ergodic Representations}
\label{Ergodic Representations Subsection}
%%%%  %%%% %%%%  %%%%  %%%%  %%%%

We now give a characterization of when an $\Aut(\M)$-recipe gives rise to an ergodic $\Aut(\M)$-invariant measure. Recall that by \cref{Ergodic = extreme} these are exactly the extreme measures in the simplex of $\Aut(\M)$-invariant measures and by \cref{Th(mu) is complete for ergodic mu}, each ergodic invariant measure has an almost sure complete and consistent theory.

\begin{proposition}
Suppose $\mu$ is an $\Aut(\M)$-invariant measure on $\Str_\Lang(\M)$. Then the following are equivalent.
\begin{itemize}
\item[(1)] $\mu$ is ergodic. 

\item[(2)] $\mu$ is dissociated. 

\item[(3)] There is an $\Aut(\M)$-recipe $\f =\<f_p\>_{p \in \qftp(\M)}$ such that 
\begin{itemize}
\item $\RandMod{\f}$ has distribution $\mu$

\item For each $p \in \nqftp(\M)$, $f_p$ does not depend on the coordinate with $\emptyset$-index. 
\end{itemize}

\end{itemize}

\end{proposition}
\begin{Proof}
First assume (3) holds. Recall \cref{S_p and alpha_p}. We know by \cref{Erdos-Renyi random structure is ergodic} that $\vartheta_\M$ is ergodic and so by \cref{Equivalence of ergodicity of mu and mu boxplus nu_M} it suffices to show that $\mu \boxplus \vartheta_\M$ is ergodic. 

Let $\e = \<e_n\>_{n \in \Nats}$ be the $\Sym{\Nats}$-recipe where,
\begin{itemize}
\item for each $n \in \Nats$, 

\item for each $p \in \qftp(\M)$ of arity $n$

\item for each $\x \in S_p$, $e_n(\x) = f_p \circ \alpha_p^{-1}(\x)$.
\end{itemize}
We then have $\RandMod{\e}$ has the same distribution as $\mu \boxplus \vartheta_\M$. But for each $n \in \Nats$ it is immediate that the value of $e_n(\x)$ is independent of the value of the coordinate indexed by $\emptyset$. Therefore by \cref{Ergodic Sym(N)-Invariant Measures} we have that $\mu \boxplus \vartheta_\M$ is ergodic and (1) holds. 

Next assume (1) holds. By \cref{Sym(N)-recipe agreeing with c^M} there is an ergodic $\Sym{\Nats}$-recipe $\e = \<e_n\>_{n \in \Nats}$ whose distribution is $\mu \boxplus \vartheta_\M$ and which agrees with $\c^\M$. We therefore have $\f \defas \<e_{\arity(p)} \circ \alpha_p\>_{p \in \nqftp(\M)}$ is an $\Aut(\M)$-recipe with distribution $\mu$ and for which all of the functions are independent of the coordinate indexed by $\emptyset$. Therefore (3) holds. 

Finally note by \cref{Ergodic Sym(N)-Invariant Measures}, $\vartheta_\M$ is dissociated. Therefore, because of how $\mu \boxplus \vartheta_\M$ was defined, $\mu$ is dissociated if and only if $\mu \boxplus \vartheta_\M$ is dissociated. But once again by \cref{Ergodic Sym(N)-Invariant Measures}, $\mu \boxplus \vartheta_\M$ is dissociated if and only if $\mu \boxplus \vartheta_\M$ is ergodic if and only if $\mu$ is ergodic. Therefore (1) is equivalent to (2).
\end{Proof}

%%%%  %%%%  %%%%  %%%%  %%%%  %%%%  %%%%  %%%%
\section{Acknowledgements}
%%%%  %%%%  %%%%  %%%% %%%%  %%%%  %%%%  %%%%

I would like to thank Cameron Freer and Rehana Patel for many helpful conversations as well as for comments on an earlier draft. I would like to thank Alex Kruckman for useful discussions. I would also like to thank Todor Tsankov for posing the question which is answered in \cref{Complete classification over structures with trivial dcl}. 

This research was facilitated by participation in the London Mathematical Society -- EPSRC Durham Symposium on \qu{Permutation Groups and Transformation Semigroups} from July 20-30, 2015 (Durham, England) as well as the \qu{Logic and Random Graphs} workshop at the Lorentz Center in from August 31 - September 4, 2015 (Leiden, Netherlands).

%% INDEX %%
%\printindex
%\clearpage

%% BIBLIOGRAPHY %%
\printbibliography

\end{document}